\documentclass[a4paper, notitlepage]{report}
\usepackage{amssymb, amsmath, amsthm}

\usepackage[english]{babel} 
\usepackage[utf8]{inputenc} 
\usepackage{textcomp}
\usepackage{caption}
\usepackage{pdfpages}
\usepackage{chngpage}
\usepackage{comment}
\usepackage{hyperref}
\usepackage{todonotes}
\usepackage{pgf,tikz}
\usetikzlibrary{arrows}
\usetikzlibrary{calc}
\usepackage{fancyhdr}
\usepackage{wallpaper}
\usepackage{makeidx}
\usepackage{titling}

\title{Entropy under flow equivalence and a classification of non-sofic $S$-gap shifts}
  \author{Peter Michael Reichstein Rasmussen}
  \date{\today}

\newcommand{\simequiv}{\mathrel {\vcenter{
  \offinterlineskip\halign{\hfil$##$\cr\noalign{\kern-1.5pt}\sim\cr\noalign{\kern-1.5pt}\sim\cr\noalign{\kern-1.5pt}\sim\cr\noalign{\kern-2.3pt}}}}}

\newcommand{\ex}[2] {
^{#1\mapsto #2}
}

\newcommand{\abs}[1]{\left\vert #1 \right\vert}

\newcommand{\orb}[1]{\mathcal O\left(#1\right)}

\newcommand{\R}{\mathbb{R}}

\newcommand{\N}{\mathbb{N}}

\newcommand{\Z}{\mathbb{Z}}

\newcommand{\A}{\mathcal A}
\newcommand{\F}{\mathcal F}
\newcommand{\X}{\mathsf X}
\newcommand{\V}{\mathcal V}
\newcommand{\E}{\mathcal E}
\newcommand{\G}{\mathcal G}
\newcommand{\Lab}{\mathcal L}

\newcommand{\FE}{\sim_{FE}}

\DeclareMathOperator{\BF}{BF}
\DeclareMathOperator{\sign}{sgn}

\newcounter{thmCounter}
\numberwithin{thmCounter}{chapter}

\theoremstyle{definition}
\newtheorem{definition}[thmCounter]{Definition}

\newtheorem{example}[thmCounter]{Example}
\newtheorem{remark}[thmCounter]{Remark}

\theoremstyle{plain}
\newtheorem{theorem}[thmCounter]{Theorem}
\newtheorem{proposition}[thmCounter]{Proposition}
\newtheorem{lemma}[thmCounter]{Lemma}
\newtheorem{corollary}[thmCounter]{Corollary}
\newtheorem{conjecture}[thmCounter]{Conjecture}

\begin{document}

\begin{titlepage}
	\ThisLRCornerWallPaper{1}{baggrund}	
	\vspace*{5.5cm}
	\noindent
	{\large\textsc{Peter Michael Reichstein Rasmussen}}\\[0.5cm]
	{\huge\textsc{Flow Equivalence of Shift Spaces}}\\[0.5cm]
	\vfill\noindent
	{\large\textsc{Bachelor Thesis in Mathematics}}\\[0.2cm]
	\noindent
	{\large\textsc{Department of Mathematical Sciences}}\\[0.2cm]
	\noindent
	{\large\textsc{University of Copenhagen}}\\[1cm]
	{\large\textsc{Advisor \\[0.2cm] {\Large Søren Eilers}}}\\[1cm]
	{\large\textsc{June 4, 2015}}
\end{titlepage}

\pagenumbering{roman}

\begin{abstract}
	We study two problems related to flow equivalence of shift spaces. The first problem, the classification of $S$-gap shifts up to flow equivalence, is partially solved with the establishment of a new invariant for the sofic $S$-gap shifts and a complete classification of the non-sofic $S$-gap shifts. The second problem is an examination of the entropy of shift spaces under flow equivalence. For a wide array of classes of shift spaces with non-zero entropy, it is shown that the entropies achievable while maintaining flow equivalence are dense in $\R^+$.
\end{abstract}
\vspace{1cm}

\renewcommand{\abstractname}{Resumé}

\begin{abstract}
Vi undersøger to problemer relaterede til strømningsækvivalens af skiftrum. Det første problem, klassifikation af $S$-interval skift op til strømningsækvivalens, bliver delvist løst med etableringen af en ny invariant for sofiske $S$-interval skift og en fuldstændig klassifikation af de ikke-sofiske $S$-interval skift. Det andet problem er en undersøgelse af entropi af skiftrum under strømningsækvivalens, og det bliver vist, at den opnåelige entropi under strømningsækvivalens er tæt i $\R^+$ for en bred vifte af skiftrum med positiv entropi. 
\end{abstract}

\newpage
\section*{Introduction}
The subject of this thesis is symbolic dynamics, an old sub-discipline of the field of dynamical systems. The main occupation of symbolic dynamicists is the analysis of shift spaces; dynamical systems $(X, \sigma)$, where $X$ is a compact subset of $\A^\Z$, for some finite set $\A$, and $\sigma\colon X\to X$ is the map subtracting one from all the indices of a point $(x_i)_{i\in \Z}$. We can see the application of the map $\sigma$ to a point of $X$ as a leap in time, just as we traditionally would in dynamical system theory, and a point itself can be understood as the oscillation of a system between a finite number of states over time. 

Flow equivalence, the specific focus of this thesis, is an equivalence relation, which allow us to obstruct the flow of time by inserting ``pauses'' at certain places in every point of a shift space. The consequences of flow equivalence are as of yet relatively unexplored, and the aim of this thesis is to close a few of the holes in our knowledge by investigating uncharted and open problems.
\\

\noindent
The thesis has been divided into four chapters as follows.

The first chapter introduces the basic theory of symbolic dynamics. We discuss shift spaces, the central object of study in symbolic dynamics; sliding block codes, the most commonly considered maps between shift spaces; shifts of finite type, in some sense the simplest class of shift spaces; and sofic shifts, a wider and more versatile class of shift spaces. As we go, the most rudimentary results needed in later chapters will be presented.

The second chapter defines flow equivalence and gives a combinatorial interpretation of the otherwise topologically defined relation. Most of the chapter is devoted to proving preliminary results for later chapters, but we also see our first classification result with respect to flow equivalence: the classification of the irreducible shifts of finite type.

The third chapter investigates the behaviour of entropy of shift spaces under flow equivalence. Not much was known beforehand in this area, though it seemed that entropy could potentially vary greatly with flow equivalence, which also turns out to be the case. We start by showing that all shift spaces of non-zero entropy are flow equivalent to shift spaces of arbitrarily small entropy. We then develop techniques for identifying shift spaces, which have flow equivalence classes with entropy dense in $\R^+$. The last part of the chapter is devoted to applying these techniques to different classes of shift spaces.

The fourth chapter ponders the open problem of the classification of the $S$-gap shifts up to flow equivalence. This problem has earlier been attacked by means of graph invariants under flow equivalence for sofic shifts, but in this chapter we develop a new technique, which by viewing flow equivalences as a map between shift spaces partly solves the problem. After reviewing previous results on flow equivalence of $S$-gap shifts, we turn to the development of our new technique, which we ultimately use to establish a stronger invariant for sofic $S$-gap shifts and completely classify the non-sofic $S$-gap shifts.

\vfill 
\renewcommand{\abstractname}{Acknowledgements}

\begin{abstract}
I would like to thank my advisor Søren Eilers for answering my many questions and for supplying me with problems to solve, theorems to prove. Further, Frederik Ravn Klausen deserves my gratitude for his comments on the manuscript.
\end{abstract}

\tableofcontents
\newpage
\pagenumbering{arabic}

\chapter{Shift spaces}
In this very first chapter we introduce the notion of a shift space, the central object of study in the field of symbolic dynamics, and account for the basic concepts and results needed in later chapters.

Some of the proofs of the chapter have been left out since we have to build up a rather extensive theory in a short presentation.

\begin{remark}
All results in this chapter may be found in the literature. 
Any result that does not have a citation is due to Lind and Marcus \cite{LM}, although the author has provided alternative proofs of Propositions \ref{pointMatchesLanguage}, \ref{languageUniqueness}, and \ref{finiteMStep}.
\end{remark}

\section{Introduction}
Let a finite set $\A$, called an \emph{alphabet}, be given and consider the set $\A^{\Z}$ of bi-infinite sequences over $\A$ together with a map $\sigma\colon \A^\Z \to \A^\Z$, which maps a point $x=(x_i)_{i\in \Z}\in \A^\Z$ to the point $(y_i)_{i\in \Z}=y=\sigma(x)$ with $y_i=x_{i+1}$ for all $i\in \Z$. The pair $(\A^\Z, \sigma)$ is called the \emph{full shift} over $\A$, and $\sigma$ is called the \emph{shift map}.

Shift spaces, which we introduce in this section, are certain subsets $X\subseteq \A^\Z$ satisfying $\sigma(X) = X$, and the name originates exactly from this invariance under the shift map. 

In topological terms, we will consider $\A$ in the discrete topology, the full shift over $\A$ will be $\A^\Z$ endowed with the product topology, and shift spaces will be the compact subsets of $\A^\Z$ that are invariant under the shift map. Note, however, that this treatment of the subject will rely mainly on a combinatorial interpretation of shift spaces and not require any prior knowledge of topology, although we will remark on certain results from a topological vantage point.

\subsection{Forbidden words and shift spaces}		
	Let $\A$ be an alphabet. A \emph{word} over $\A$ is a finite sequence over $\A$, written $w=w_1\dots w_n$ for $w_i\in \A$. We say that $w$ has length $\abs{w} = n$ and let the empty word of length $0$ be denoted $\epsilon$. If $w=w_1\dots w_n$ and $v=v_1\dots v_m$ are words over $\A$ then $wv=w_1\dots w_nv_1\dots v_m$ denotes their concatenation, and for $k\in \N$, $w^k$ is the concatenation of $k$ copies of $w$.

For a point $x=(x_i)_{i\in\Z}$ of the full shift $\A^\Z$, we let $x_{[i, j]}, i\leq j$, signify the word $w=x_i\dots x_j$ and say that $w$ \emph{occurs} in $x$. Similarly, for a word $w=w_1\dots w_n$ over $\A$ and $i, j\in \N$ with $1\leq i\leq j\leq n$,  we denote by $w_{[i,j]}$ the \emph{subword} $u=w_i\dots w_j$ of $w$ and say that $u$ is a factor of $w$.
\\

The canonical way of defining a shift space combinatorially is by the words that do not occur in any of its points. Let $\A$ be a finite alphabet, $\F$ a set of words over $\A$, and $\X_\F$ the set of points $x\in\A^\Z$ such that no word of $\F$ occurs in $x$. The set $\F$ is called a set of \emph{forbidden words} for $\X_\F$.
\begin{definition}
	Let $\A$ be an alphabet. A set $X\subseteq\A^\Z$ is a \emph{shift space} if there is a set of forbidden words $\F$ over $\A$ such that $X=\X_\F$. If an arbitrary shift space $X$ is given, we denote its alphabet by $\A(X)$, and the shift map restricted to $X$ by $\sigma_X$.
	\end{definition}

\begin{remark}
	From the definition it follows that shift spaces are indeed invariant under the shift map, since the words occurring in a point are invariant under the shift map. It is worth noting that not all shift invariant subsets of a full shift are shift spaces.
\end{remark}

\begin{example}
Let $\A $ be an alphabet and $\F = \emptyset$. Then $\X_\F$ is simply the full shift over $\A$. When the alphabet is of the form $\A = \{0, 1, \dots, n-1\}$ the shift space is known as the \emph{full $n$-shift} and is written $\X_{[n]}$.
\end{example}

\begin{example}\label{sGapShiftEx}
Let $S\subseteq \N$ be non-empty. The \emph{S-gap shift} $\X(S)$ has alphabet $\A=\{0, 1\}$ and consists of all bi-infinite binary sequences such that the number of 0's between any two consecutive 1's belong to $S$.

 More precisely $\X(S)=\X_\F$ where
\begin{equation*}
	\F = \begin{cases}
			 \{10^k1\mid k\not\in S \text{ and } k<\max{S}\}\cup \{0^{1+\max S}\}, &\text{if $S$ is finite}\\
 	         \{10^k1\mid k\not\in S \}, &\text{if $S$ is infinite.}
 		 \end{cases}
\end{equation*}
If $S$ is infinite, the point $0^\infty$ belongs to $\X(S)$ since forbidding any word of the form $0^k, k\in \N$ would forbid words of the form $10^s1, s\in S$ for $s\geq k$.
\end{example}

\begin{definition}
	A point $x$ is \emph{periodic} with \emph{period} $n\in \N$ if $\sigma^n(x)=x$, and $x$ is periodic if such an $n$ exists. The least $n\in \N$ satisfying $\sigma^n(x)=x$ is called the \emph{least period} of $x$.
\end{definition}
\noindent
If $x$ is periodic it must have the form $x = \dots www.www\dots$ for some word $w$, and we write this as $x=w^\infty$.

A significant property of the least period of a point, the proof of which is omitted for brevity, is the following.

\begin{proposition}\label{leastPeriod}
	The least period of a point $x$ divides every other period of $x$.
\end{proposition}

\subsection{Languages}
Instead of focusing on the forbidden words of a shift, we will often want to discuss the allowed ones. We will say that a word $w$ \emph{occurs} in a shift space $X$, if there is a point $x\in X$ such that $w$ occurs in $x$.
\begin{definition}
Let $X$ be a shift space. We denote by $B_n(X)$ the set of all words of length $n$ that occur in $X$. The \emph{language} of $X$ is the set $B(X)=\bigcup_{i=0}^\infty B_i(X)$ of all words that occur in $X$.
\end{definition}
\noindent
The language of a shift space gives a convenient way to check whether or not a point belongs to the shift space.

\begin{proposition}\label{pointMatchesLanguage}
	Let $X$ be a shift space, $\A$ its alphabet, and $x\in \A^\Z$ a point. Then $x\in X$ if and only if every word occurring in $x$ belongs to $B(X)$.
\end{proposition}
\begin{proof}
	If $x\in X$ then every word occurring in $x$ is contained in $B(X)$ by definition. Conversely, suppose that every word occurring in $x$ is contained in $B(X)$, and let $\F$ be a set of forbidden words over $\A$ such that $X=\X_\F$. Since $B(X)\cap \F=\emptyset$ no word of $\F$ occurs in $x$ and we get $x\in \X_\F=X$.
\end{proof}
\noindent
Contrary to a set of forbidden words and its corresponding shift space, a shift space and its language uniquely identify each other.

\begin{proposition}\label{languageUniqueness}
	Two shift spaces are equal if and only if their languages are equal.
\end{proposition}
\begin{proof}
	If two shift spaces are equal, then they clearly have the same language. Conversely, if two shift spaces $X$ and $Y$ have the same language and $x\in X$, then for every word $w$ occurring in $x$ we have $w\in B(X)=B(Y)$, so by Proposition \ref{pointMatchesLanguage}, $x\in Y$. Thus, $X\subseteq Y$ and by a symmetrical argument, $Y\subseteq X$.
\end{proof}
\noindent
There is a certain class of shift spaces that play an important role in the theory of symbolic dynamics. This is due to a property that allows us to ``glue'' together any two words of the shift space with a connecting word.
\begin{definition}
	A shift space $X$ is \emph{irreducible} if for every ordered pair $u, v\in B(X)$ there is a word $w$ such that $uwv\in B(X)$.
\end{definition}

\subsection{Sliding block codes}
If we consider two shift spaces $X$ and $Y$ in a topological context, it is reasonable to study continous maps $\varphi\colon X\to Y$ that commutes with the shift map, i.e. $\varphi(\sigma_X(x))=\sigma_Y(\varphi(x))$ for all $x\in X$. Since we do not rely on topology in this treatment, we will instead use a combinatorial interpretation of such a map, the sliding block code.

\begin{definition}
	Let $X$ be a shift space and $\A$ an alphabet. A \emph{block map} or \emph{$(m+n+1)$-block map} is a function $\Phi\colon B_{m+n+1}(X)\to \A$. If $w=a_1\dots a_r$ with $r\geq m+n+1$ we will write
	\begin{equation*}
		\Phi(w) = \Phi(a_1\dots a_{m+n+1})\Phi(a_2\dots a_{m+n+2})\dots \Phi(a_{r-m-n}\dots a_r).
	\end{equation*}
\end{definition}
\noindent A sliding block code is  the map we get from ``sliding'' a block map down a point in the following way.
\begin{definition}
	Let $X$ be shift spaces and $\Phi\colon B_{m+n+1}(X)\to \A$ a block map. The \emph{sliding block code} induced by $\Phi$ is the map $\Phi_\infty^{[m, n]} = \phi\colon X\to \A^\Z$ satisfying
	\begin{equation*}
		(\phi(x))_i = \Phi(x_{i-m}\dots x_{i+n})
	\end{equation*}
	for all $x\in X$ and $i\in \Z$. We call $m$ the \emph{memory} and $n$ the  \emph{anticipation} of $\phi$ and say that $\phi$ is an \emph{$(m+n+1)$-block code}. If $Y\subseteq \A^\Z$ is a shift space with $\phi(X)\subseteq Y$ we write $\phi\colon X\to Y$.
\end{definition}
\noindent
We see from the definition that a sliding block code does indeed commute with the shift map.

\begin{example}
	Let $X$ be the full 2-shift and let $\Phi\colon B_2(X)\to \{0, 1\}$ be given by $\Phi(w_1w_2) = w_2$. Then the map $\phi = \Phi^{[0, 1]}_\infty\colon X\to X$ is the shift map $\sigma_X$.
\end{example}
\noindent
An injective sliding block code $\phi\colon X\to Y$ is called an \emph{embedding} of $X$ into $Y$ and we say that $X$ \emph{embeds} into $Y$ if such a map exists. Similarly, a surjective sliding block code $\phi\colon X\to Y$ is called a \emph{factor code} from $X$ onto $Y$ and we say that $Y$ is a \emph{factor} of $X$ if there is a factor code from $X$ to $Y$. 

We say that a bijective sliding block code $\phi\colon X\to Y$ \emph{has an inverse}, if the inverse function $\phi^{-1}$ is a sliding block code, and in that case we call $\phi$ a \emph{conjugacy} from $X$ to $Y$. 
\begin{definition}
	Two shift spaces $X$ and $Y$ are \emph{conjugate}, written $X\cong Y$, if there is a conjugacy from $X$ to $Y$.
\end{definition}
\noindent
Conjugacy defines an equivalence relation on shift spaces, and two conjugate shift spaces share most every property. From a topological perspective this is intuitive as a conjugacy is a homeomorphism that commutes with the shift map.

We now bring two important propositions. They have technical combinatorial proofs, which we will skip for brevity, but their topological proofs are elegant and will be sketched here for sake of intuition. For the combinatorial proofs, see Theorem 1.5.13 and 1.5.14 of Lind and Marcus \cite{LM}.

\renewcommand*{\proofname}{Sketch of proof}
\begin{proposition}\label{imageUnderBlockCode}
Let $X$ be a shift space, $\A$ some alphabet, and $\phi\colon X\to \A^\Z$ a sliding block code. Then the image $\phi(X)$ is a shift space.
\end{proposition}
\begin{proof}
	Shift spaces are the compact subsets of $\A^\Z$ and the image of a compact space under a continuous function is compact.
\end{proof}
\begin{proposition}\label{bijHasInverse}
Every bijective sliding block code has an inverse.
\end{proposition}
\begin{proof}
	Shift spaces are both compact and metric, and any bijective continuous function between such spaces is a homeomorphism. 
\end{proof}
\noindent

\renewcommand*{\proofname}{Proof}
\noindent
An example of the conservation of properties under conjugacy is the following.
\begin{lemma}
Let $\phi\colon X\to Y$	be a conjugacy and $x\in X$ be periodic with least period $p$. Then $\phi(x)$ also has least period $p$.
\end{lemma}
\begin{proof}
	Let $x\in X$ have least period $p$ and $y=\phi(x)$ have least period $q$. Then 
	\begin{equation*}
		\sigma^p_Y(y) = \sigma^p_Y(\phi(x)) = \phi(\sigma^p_X(x)) = \phi(x)=y,
	\end{equation*}
	so $q\mid p$ by Proposition \ref{leastPeriod} as $p$ is a period of $y$.
	Similarly, $p\mid q$ as $\phi$ has an inverse, which is a sliding block code. It follows that $q=p$.
\end{proof}

\section{Shifts of finite type}
A central class of shift spaces is the shifts of finite type, which consists of all shift spaces that can be specified by a finite set of forbidden words. This section is dedicated to the study of such shift spaces.
\begin{definition}
	A \emph{shift of finite type} is a shift space $X$ such that there is a finite set of forbidden words $\F$ with $X=\X_\F$.
\end{definition}
\noindent
A very practical property of a shift of finite type is that for some $M\in \N_0$ we can verify that a point $x$ belongs to $X$ only by checking all words of length $M+1$ occurring in $x$.
\begin{definition}
	A shift space $X$ is \emph{$M$-step} for some $M\in \N$ if there is a set of forbidden words $\F$ all of length $ M+1$ such that $X=\X_\F$.
\end{definition}
\begin{proposition}\label{finiteMStep}
A shift space $X$ is of finite type if and only if it is $M$-step for some $M\in \N$.	
\end{proposition}
\begin{proof}
	If $X$ is $M$-step then it is of finite type, since there are only finitely many words of length $M+1$ over $\A(X)$.
	
	Conversely, if $X$ is of finite type there is a finite set of forbidden words $\F$ with $X=\X_\F$. If $\F=\emptyset$ then $X$ is vacuously $M$-step for all $M$. Else, let $M+1=\max_{w\in \F}\abs{w}$ and create a new set of forbidden words $\F'$ by replacing every $w\in \F$ with all words over $\A(X)$ of length $M+1$ that has $w$ as a factor. Then $\F'$ satisfies $\abs{w}=M+1$ for all $w\in \F'$ and $X=\X_{\F'}$.
\end{proof}
\noindent
The class of shifts of finite type is closed under conjugacy. We omit the proof for brevity, but it can be found as Theorem 2.1.10 of Lind and Marcus \cite{LM}.
\begin{theorem}
	If $X$ and $Y$ are conjugate shift spaces then $X$ is of finite type if and only if $Y$ is of finite type.
\end{theorem}
\noindent

\subsection{Shifts of finite type represented by graphs}
We can conveniently represent any shift of finite type as a directed graph, which allows us to employ the powerful tools of graph theory and linear algebra in our examination of them. In this section we will introduce some basic graph theory and link it to shifts of finite type.

\begin{definition}
	A \emph{directed graph} $G$ is a pair of finite sets $(\mathcal V,\mathcal E)$ with two maps $s, t\colon \mathcal E\to \mathcal V$. We say that $\V$ and $\E$ are the vertices and edges of the graph respectively, and for any edge $e\in \E$ we say that $e$ starts at $s(e)$ and terminates at $t(e)$. 
	If a graph $G$ is given we will denote its vertices and edges by $\V(G)$ and $\E(G)$ respectively.
	
	A \emph{subgraph} $H\subseteq G$ is a graph with $\V(H)\subseteq \V(G)$ and $\E(H)\subseteq \E(G)$ such that all edges of $\E(H)$ start and terminate at vertices of $\V(H)$.
\end{definition}

\noindent
Note that directed graphs can have multiple edges between the same two vertices and edges that start and terminate at the same vertex. Figure \ref{directedGraphs} illustrates two examples of directed graphs.
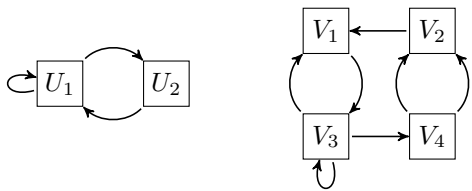
\begin{figure}
\begin{center}

\begin{tikzpicture}
  [bend angle=45,
   knude/.style = {circle, inner sep = 0pt},
   vertex/.style = {circle, draw, minimum size = 1 mm, inner sep =
      0pt,},
   textVertex/.style = {rectangle, draw, minimum size = 6 mm, inner sep =
      1pt},
   to/.style = {->, shorten <= 1 pt, >=stealth', semithick}, scale=0.7]
  
  \node[knude] (trans) at (5,0) {} ;

  \node[textVertex] (u) at (0,-1) {$U_1$};
  \node[textVertex] (v) at (2,-1) {$U_2$};

  \node[textVertex] (P1) at ($(u)+(0,1)+(trans)$) {$V_1$};
  \node[textVertex] (P2) at ($(v)+(0,1)+(trans)$) {$V_2$};
  \node[textVertex] (P3) at ($(0,-2)+(trans)$) {$V_3$};
  \node[textVertex] (P4) at ($(2, -2)+(trans)$) {$V_4$};

  \draw[to, bend left=45] (u) to node[auto] {} (v);
  \draw[to, bend left=45] (v) to node[auto] {} (u);
  \draw[to, loop left] (u) to node[auto] {} (u);
  \draw[to, bend left=45] (P1) to node[auto] {} (P3);
  \draw[to, bend left=45] (P3) to node[auto] {} (P1);
  \draw[to, bend left=45] (P4) to node[auto] {} (P2);  
  \draw[to, bend right=45] (P4) to node[auto] {} (P2);
  \draw[to] (P3) to node[auto] {} (P4);
  \draw[to] (P2) to node[auto, swap] {} (P1);
  \draw[to, loop below] (P3) to node[auto] {} (P3);
\end{tikzpicture}
\caption{Examples of directed graphs.}\label{directedGraphs}
\end{center}
\vspace{-0.5cm}
\end{figure}

\begin{definition}
	Let $G$ be a graph. A \emph{path} on $G$ is a sequence $\pi=e_1\dots e_n, e_i\in \E(G)$ such that $t(e_i)=s(e_{i+1})$ for all $1\leq i <n$. For two vertices $V_1, V_2\in \V(G)$ we say that $V_1$ is \emph{connected} to $V_2$ if there is a path $\pi=e_1\dots e_n$ on $G$ with $s(e_1)=V_1$ and $t(e_n)=V_2$.
\end{definition}
\noindent
Now, given a graph $(\V, \E)$ we can define a shift space with $\E$ as its language and where all points are bi-infinite walks on the graph.
\begin{definition}
	Let $G$ be a graph. The \emph{edge shift} of $G$ has alphabet $\E(G)$, is denoted by $\X_G$, and defined by
	\begin{equation*}
		\X_G = \{(e_i)_{i\in \Z}\mid \forall i\in \Z\colon t(e_i)=s(e_{i+1}) \}.
	\end{equation*}
\end{definition}
\noindent
A set of forbidden words for $\X_G$, obtained from the restrictions of the graph, is
\begin{equation*}
	\F = \{ e_1e_2 \mid e_1, e_2\in \E(G) \text{ and } t(e_1)\neq s(e_{2}) \},
\end{equation*}
so $\X_G$ is a shift space of finite type.

\begin{definition}
	Let $G$ be a graph. We say that a vertex $V$ of $G$ is \emph{stranded} if either there is no edge terminating at $V$ or no edge starting at $V$.
	
	 A graph is called \emph{essential} if none of its vertices are stranded. 
\end{definition}
\noindent
As no bi-infinite walk on a graph can include a stranded edge, we can always assume, in the context of edge shifts, that a graph is essential, since adding stranded edges to its graph will not alter its edge shift. When a graph is essential there is a bijective correspondence between paths on the graph and the words of its edge shift as all paths on the graph can be extended to bi-infinite walks on the graph.

As promised, we will see that all shifts of finite type can be represented, up to conjugacy, by an edge shift.
\begin{definition} \label{higherBlockDef} 
	Let $X$ be a shift space, $m\in \N$ be given, and the block map $\Phi\colon B_m(X)\to B_m(X)$ be the identity. Then the \emph{higher block shift} $X^{[m]}$ has alphabet $\A(X^{[m]}) = B_m(X)$ and is given by the image $X^{[m]}= \Phi_\infty^{[0, m-1]}(X)$.
\end{definition}
\noindent
Since the map in the above definition is injective, $X^{[m]}$ and $X$ are conjugate.

\begin{theorem}\label{finiteTypeIsGraph}
	A shift space $X$ is of finite type if and only if it is conjugate to an edge shift.
\end{theorem}
\begin{proof}
	We already noted that all edge shifts are of finite type, so we need only show that every shift of finite type is conjugate to some edge shift.
	
	Let $X$ be $M$-step with a collection $\F$ of forbidden words of length $M+1$, and define a graph $G$ with vertex set $\V=B_M(X)$, edge set $\E = B_{M+1}(X)$, and maps $s, t\colon B_{M+1}(X)\to B_{M}(X)$ satisfying $s(w)=w_{[1, M]}$ and $t(w)=w_{[2,M+1]}$. Then
	\begin{equation*}
		\X_G = \{ (w_i)_{i\in \Z}\in B_{M+1}(X)^\Z \mid \forall i\in \Z \colon w_{i\,[2,M+1]} = w_{i+1\,[1,M]} \}.
	\end{equation*}
	
	Let $X^{[M+1]}$ be the higher block shift of $X$. As all words of $\F$ have length $M+1$  a set of forbidden words for $X^{[M+1]}$ is given by 
	\begin{equation*}
		\F' = \F \cup \{ w_1w_2 \mid w_1, w_2\in B_{M+1}(\A(X)^{\Z}) \text{ and } w_{1\,{[2, M+1]}}\neq w_{2\,[1, M]} \}.
	\end{equation*}
	As it is clear that $\F'$ is also a set of forbidden words for $\X_G$, we have $X^{[M+1]}=\X_G$, and the result now follows since $X^{[M+1]}\cong X$.	
\end{proof}
\noindent
So in a sense, all shifts of finite type are essentially edge shifts, and it is fairly easy, as the above proof illustrates, to construct an edge shift conjugate to a given shift of finite type. 

We have a concept of irreducibility for directed graphs, and this property turns out to be equivalent with irreducibility of the associated edge shift.
\begin{definition}
		 A graph $G$ is called \emph{irreducible} if for all ordered pairs $V_1, V_2\in \V(G)$ of vertices, $V_1$ is connected to $V_2$.	
\end{definition}
\begin{proposition}\label{irreducibleMatrixVsGraph}
Let $G$ be an essential graph. The edge shift $\X_G$ is irreducible if and only if the graph $G$ is irreducible.	
\end{proposition}
\begin{proof}
	Assume that $G$ is irreducible and let $u=e_1\dots e_n, v=f_1\dots f_m\in B(\X_G)$. Then there must be a path $w=g_1\dots g_k$ from $t(e_n)$ to $s(f_1)$, so $uwv$ is a path on $G$ and $uwv\in B(\X_G)$.
	
	On the other hand, if $\X_G$ is irreducible and $V_1, V_2$ are vertices of $G$, then there are edges $u$ and $v$ that terminates and starts at $V_1$ and $V_2$ respectively. By irreducibility of $\X_G$ there is a $w\in B(\X_G)$ such that $uwv\in B(\X_G)$, so $uwv$ is a path on $G$ from $V_1$ to $V_2$.
\end{proof}

\subsection{Classification of shifts of finite type}
A major aspiration of symbolic dynamics is to decide, given two shift spaces, whether or not they are conjugate. In general this is very difficult, but in the case of shifts of finite type the question has somewhat been resolved by the aid of linear algebra.

\begin{definition}
	Let $G$ be a graph and number the vertices $\V(G)=\{ 1, \dots, n \}$. The \emph{adjacency matrix} of $G$ is the $n\times n$ matrix $A$ such that each entry $A_{ij}$ is the number of edges starting at vertex $i$ and terminating at vertex $j$.
	
	To every non-negative integer matrix $A$ we will assign a graph $G$ having adjacency matrix $A$. Further, if $A$ is the adjacency matrix for $G$ we will let $\X_A$ denote the edge shift $\X_G$.
\end{definition}
\begin{example}
The adjacency matrices of the graphs of Figure \ref{directedGraphs} are
\begin{equation*}
	\begin{pmatrix}
		1 & 1 \\
		1 & 0
	\end{pmatrix}
	\,\,\,\, \text{ and }\,\,\,\, 
	\begin{pmatrix}
		0 & 0&1&0 \\
		1 & 0 & 0&0\\
		1&0&1&1\\
		0&2&0&0
	\end{pmatrix}.
\end{equation*}
\end{example}
\noindent
Now, consider the following relations on matrices.
\begin{definition}
	Let $A, B$ be square non-negative integer matrices. An \emph{elementary equivalence from} $A$ \emph{to} $B$ is a pair $(R, S)$ of (not necessarily square) non-negative integer matrices satisfying 
	\begin{equation*}
		A=RS \text{ and } SR=B.
	\end{equation*}
If there exists an elementary equivalence from $A$ to $B$ we say that they are \emph{elementary equivalent} and write $A\simequiv B$.

A \emph{strong shift equivalence from} $A$ \emph{to} $B$ \emph{of lag} $\ell$ is a sequence of elementary equivalences 
\begin{equation*}
	A\simequiv A_1, A_1\simequiv A_2, \dots, A_{\ell-1}\simequiv B.
\end{equation*}
We say that $A$ and $B$ are \emph{strong shift equivalent}, writing $A\approx B$, if there exists a strong shift equivalence from $A$ to $B$ of lag $\ell$ for some $\ell$.
\end{definition}

\noindent

\noindent
Elementary equivalence is not an equivalence relation as it lacks transitivity, but strong shift equivalence is, and it turns out that we can determine conjugacy from strong shift equivalence.
\begin{theorem}[Williams \cite{FiniteClassification}]
	Let $A$ and $B$ be square non-negative integer matrices. The edge shifts $\X_A$ and $\X_B$ are conjugate if and only if $A$ and $B$ are strong shift equivalent.
\end{theorem}
\noindent
Deciding whether or not two matrices are strong shift equivalent is often easier than finding a direct conjugacy between their edge shifts, but it is still no simple task. In fact, no deterministic procedure for doing so has been found yet.

\section{Sofic shifts}
In this section we introduce a new class of shift spaces, the sofic shifts. The sofic shifts are defined by labeled graphs and turn out to constitute a much more versatile class of shift spaces than the shifts of finite type.

\begin{definition}
	Let $\A$ be an alphabet. A \emph{labeled graph} $\G$ is a pair $(G, \Lab)$ consisting of a graph $G$ and a \emph{labelling} $\Lab \colon \E(G)\to \A$ assigning to each edge of $G$ a symbol from $\A$. The \emph{sofic shift} presented by $\G$ is given by
	\begin{equation*}
		\X_\G = \{ \Lab_\infty^{[0, 0]} (x) \mid x\in \X_G \},
	\end{equation*}
	and we say that a shift space $X$ is sofic, if it has a \emph{presentation} as a labeled graph.
	A path $\pi = e_1\dots e_n$ on $G$ with $w = \Lab(\pi)$ is said to be a \emph{presentation} of $w$ in $\G$.
\end{definition}

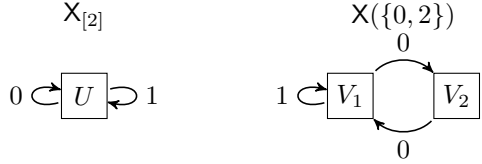
\begin{figure}
\begin{center}
\begin{tikzpicture}
  [bend angle=45,
   knude/.style = {circle, inner sep = 0pt},
   vertex/.style = {circle, draw, minimum size = 1 mm, inner sep =
      0pt,},
   textVertex/.style = {rectangle, draw, minimum size = 6 mm, inner sep =
      1pt},
   to/.style = {->, shorten <= 1 pt, >=stealth', semithick}, scale=0.7]
  
  \node[knude] (trans) at (5,0) {} ;

  \node[textVertex] (u) at (0,0) {$U$};

  \node[textVertex] (P1) at ($(u)+(trans)$) {$V_1$};
  \node[textVertex] (P2) at ($(2,0)+(trans)$) {$V_2$};

  \node[knude] (full)  at (0, 1.5) {$\X_{[2]}$};
  \node[knude] (S) at ($(full)+(1,0)+(trans)$) {$\X(\{0, 2\})$};

  \draw[to, loop left=45] (u) to node[auto] {0} (u);
  \draw[to, loop right=45] (u) to node[auto] {1} (u);

  \draw[to, bend left=45] (P1) to node[auto] {0} (P2);
  \draw[to, bend left=45] (P2) to node[auto] {0} (P1);
  \draw[to, loop left] (P1) to node[auto] {1} (P1);
\end{tikzpicture}
\caption{Presentations  of the full 2-shift and the $S$-gap shift with $S=\{0, 2\}$ as labeled graphs.}\label{soficShifts}
\end{center}
\vspace{-0.5cm}
\end{figure}

\noindent
Two examples of presentations of sofic shifts are shown in Figure \ref{soficShifts}. We start out by making sure that sofic shifts are in fact shift spaces.

\begin{proposition}
Any sofic shift is a shift space.	
\end{proposition}
\begin{proof}
	Let $\G=(G, \Lab)$ be a presentation of a sofic shift with language $\A$. Then $\Lab$ induces the 1-block code $\Lab_\infty^{[0, 0]}\colon \X_G\to \A^\Z$, and its image $\X_\G=\Lab_\infty^{[0, 0]}(\X_G)$ is a shift space by Proposition \ref{imageUnderBlockCode}.
\end{proof}
\noindent
One of the motivations for (and definitions of) the class of sofic shifts is that it is the smallest class of shift spaces that includes shifts of finite type and is closed under factor codes.

\begin{proposition}\label{factorOfFinite}
A shift space is sofic if and only if it is a factor of a shift of finite type.	
\end{proposition}
\begin{proof}
	Let $\X_\G$ be a sofic shift with the presentation $\G=(G, \Lab)$. Then $\X_\G = \Lab_\infty^{[0, 0]}(X_G)$ by definition, so it is a factor of $\X_G$, a shift of finite type.
	
	Conversely, suppose that $X$ is a shift space such that there is an $M$-step shift $Y$ and an $(m+n+1)$-block code $\phi\colon Y\to X$ induced by $\Phi:B_{m+n+1}\to \A(X)$, which is a factor map. 
	Since $Y$ is $M'$-step for every $M'\geq M$ and for every $m'\geq m$ there is a function identical to $\phi$ which is an $m'+n+1$-block code, we can assume that $m+n+1=M+1$. 
	
	From the proof of Theorem \ref{finiteTypeIsGraph} we know that $Y^{[M+1]}=Y^{[m+n+1]}$ is an edge shift. Let $G$ be the graph of $Y^{[M+1]}$ such that $Y^{[M+1]}=\X_G$, and $\psi=\Psi_\infty^{[0,M]}\colon Y\to Y^{[M+1]}$ be the conjugacy induced by the identity map $\Psi\colon B_{M+1}(Y)\to B_{M+1}(Y)$, as seen in Definition \ref{higherBlockDef}. Define $\Lab\colon B_{m+n+1}(Y)\to \A(X)$ by $\Lab(w)=\Phi(w)$. Then $\Lab$ induces a 1-block code $\Lab_\infty^{[0, 0]}\colon Y^{[M+1]}\to X$, and we will show that $\G=(G, \Lab)$ is a presentation of $X$, proving that $X$ is sofic.
	
	For all $y\in Y$ we have
	\begin{equation*}
	\phi(y)_{[i]} = \Phi(y_{[i-m, i+n]})\, \, \text{ and }\,\,\Lab_\infty^{[0,0]}(\psi(y))_{[i]} = \Lab(\psi(y)_{[i]})=\Phi(y_{[i, i+m+n]}),
	\end{equation*}
	 which implies $\phi(y) = \Lab_\infty^{[0, 0]}(\psi(\sigma^{-m}(y)))$ for all $y\in Y$. Since $\sigma$ and $\psi$ are both conjugacies this yields 
	\begin{equation*}
		X=\phi(Y)= \Lab_\infty^{[0, 0]}(\psi(\sigma^{-m}(Y)))=\Lab_\infty^{[0, 0]}(Y^{[m+n+1]}) =\Lab_\infty^{[0,0]}(\X_G),
	\end{equation*}  
	 so $\G$ is a presentation of $X$.
\end{proof}
\noindent
From the proposition it is now clear that the class of sofic shifts is indeed defined by the aforementioned properties.
\begin{corollary}
All shifts of finite type are sofic shifts, and the class of sofic shifts is closed under factor codes.
\end{corollary}
\begin{proof}
	First, every shift space is a factor of itself, so clearly all shifts of finite type are sofic. Second, if $\varphi\colon Y\to X$ is a factor code and $Y$ is sofic, then there is a factor code $\psi\colon Z\to Y$ for a shift of finite type $Z$. Now, $\varphi\circ \psi$ is a factor code mapping $Z$ onto $X$, so $X$ is sofic.
\end{proof}
\noindent 
A sofic shift can be presented on a standard form, the minimal right-resolving presentation. We will define it here and present a couple of results whose proofs we shall omit for brevity. See Figure \ref{figure:presentations} for examples.
\begin{definition}
	A labeled graph $(G, \Lab)$ is \emph{right-resolving} if for every vertex $V\in \V$ all edges starting at $V$ have different labels. A \emph{right-resolving presentation} of a sofic shift is a right-resolving labeled graph presenting it. A \emph{minimal right-resolving presentation} is a right-resolving presentation with a minimal number of vertices.
\end{definition}

\begin{proposition}
Every sofic shift has a minimal right-resolving presentation.	
\end{proposition}

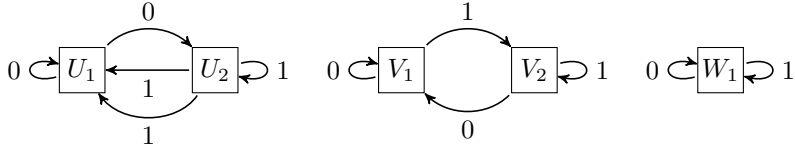
\begin{figure}
\begin{center}
\begin{tikzpicture}
  [bend angle=45,
   knude/.style = {circle, inner sep = 0pt},
   vertex/.style = {circle, draw, minimum size = 1 mm, inner sep =
      0pt,},
   textVertex/.style = {rectangle, draw, minimum size = 6 mm, inner sep =
      1pt},
   to/.style = {->, shorten <= 1 pt, >=stealth', semithick}, scale=0.7]
  
  \node[knude] (trans) at (6,0) {} ;
  
  \node[textVertex] (V1) at (0,0) {$U_1$};
  \node[textVertex] (V2) at (2.5,0) {$U_2$};

  \node[textVertex] (P1) at ($(u)+(trans)$) {$V_1$};
  \node[textVertex] (P2) at ($(2.5,0)+(trans)$) {$V_2$};

  \node[textVertex] (W1) at ($(u)+(trans)+(trans)$) {$W_1$};

  \draw[to, loop left=45] (V1) to node[auto] {0} (V1);
  \draw[to, loop right=45] (V2) to node[auto] {1} (V2);
  \draw[to, bend left=45] (V1) to node[auto] {0} (V2);
  \draw[to] (V2) to node[auto] {1} (V1);
  \draw[to, bend left=55] (V2) to node[auto] {1} (V1);

  \draw[to, bend left=45] (P1) to node[auto] {1} (P2);
  \draw[to, bend left=45] (P2) to node[auto] {0} (P1);
  \draw[to, loop left] (P1) to node[auto] {0} (P1);
  \draw[to, loop right] (P2) to node[auto] {1} (P2);
  
  \draw[to, loop left] (W1) to node[auto] {0} (W1);
  \draw[to, loop right] (W1) to node[auto] {1} (W1);
\end{tikzpicture}
\caption{Three presentations of the full 2-shift. From the left the presentations are: not right-resolving, right-resolving, and minimal right-resolving.}\label{figure:presentations}
\end{center}
\vspace{-0.5cm}
\end{figure}

\begin{definition}
	Let $\G=(G, \Lab)$ be a labeled graph and $V\in \V(G)$ be a vertex. Then the \emph{follower set} of $V$ is the set
	\begin{equation*}
		F(V) = \{ \Lab(e_1\dots e_n)\mid e_1\dots e_n \text{ is a path on $G$ and } s(e_1) = V \}.
	\end{equation*}
	We say that $\G$ is \emph{follower-separated} if for all different $V_1, V_2\in \V(G)$  we have $F(V_1)\neq F(V_2)$.
\end{definition}

\begin{proposition} \label{followerSeparated}
	If $\G$ is a minimal right-resolving presentation of a sofic shift, then it is follower-separated.
\end{proposition}

\noindent
In some shift spaces there are words that act as ``hinges'' and allow us to combine other words very flexibly. This will be pivotal in Chapter 3.

\begin{definition}
	Let $X$ be a shift space and $w\in B(X)$. We say that $w$ is \emph{intrinsically synchronising} for $X$ if whenever $vw, wu\in B(X)$ for some words $v, u\in B(X)$, then $vwu \in B(X)$.
\end{definition}
\noindent
A nice property of sofic shifts is that every word can be extended to an intrinsically synchronising word.
\begin{definition}
	Let $X$ be a sofic shift with representation $\G=(G, \Lab)$. We say that a word $w\in B(X)$ \emph{focuses to $V\in \V(G)$} if every path in $\G$ presenting $w$ terminates at $V$. 
\end{definition}

\begin{lemma}\label{focusesThenIntrins}
	Let $\G=(G, \mathcal L)$ be a labeled graph and $w\in B(\X_\G)$. If $w$ focuses to a vertex $V\in \V(G)$ then $w$ is intrinsically synchronising for $\X_\G$.
\end{lemma}
\begin{proof}
	Let words $u, v$ be given such that $vw, wu\in B(\X_\G)$. Any path representing $vw$ in $\G$ must terminate at $V$, and since there is a path representing $wu$ in $\G$ there is a path representing $u$ starting at $V$. Thus, there is a path representing $vwu$ in $\G$. 
\end{proof}

\begin{proposition}\label{extendToIntrin}
For a sofic shift $X$ any word $w\in B(X)$ can be extended to the right to an intrinsically synchronising word $wv\in B(X)$.
\end{proposition}
\begin{proof}
	Let $\G=(G, \Lab)$ be a minimal right-resolving presentation of $X$. We will show that we can construct a word $v$ such that $wv$ focuses to a vertex in $G$. By Lemma \ref{focusesThenIntrins} this is sufficient. For every $u\in B(X)$ let $T(u)$ denote the set of vertices that has a path labeled $u$ terminating at them.
	
	If $\abs{T(w)}=1$ we are done, so assume $\abs{T(w)}>1$. Let $V_1, V_2\in T(w)$ be different. Since $\G$ is follower-separated by Proposition \ref{followerSeparated}, there is a word 
	$v_1\in F(V_1)$ with $v_1\not\in F(V_2)$. As $\G$ is right-resolving there is at most one path representing $v_1$ in $\G$ starting at each member of $T(w)$, and there is none starting at $V_2$, so $\abs{T(wv_1)}<\abs{T(w)}$. Continuing this procedure we can find words $v_1, \dots, v_n$ such that $\abs{T(wv_1\dots v_n)}=1$. Setting $v=v_1\dots v_n$ completes the proof.
\end{proof}

\chapter{Flow equivalence}
Since establishing conjugacy between shift spaces is very difficult in general, it can seem inviting to consider weaker equivalence relations. One such relation is that of flow equivalence, which has found uses in ergodic theory and the theory of C*-algebras.

After a formal introduction to flow equivalence in the first section, we present some results in the second section, which will not only  be crucial in later chapters, but also gives a sense as to how we can manipulate a shift space while preserving flow equivalence. The final section brings a presentation of a classification result by John Franks, which gives a simple, complete invariant for deciding flow equivalence between irreducible shifts of finite type.

\begin{remark}
All results in this chapter can be found in the literature, though several are here presented with alternative formulations or proofs. The results of Section 2.3 have been assembled from various sources.
\end{remark}

\section{Introduction to flow equivalence}
Flow equivalence is originally a topological relation, but here we will only state the topological definition for sake of completeness. To replace it we introduce a combinatorial interpretation of flow equivalence, which is more workable in the context of this thesis.

\subsection{Flow equivalence in terms of topology}
We recall that for some alphabet $\A$, a shift space $X$ over $\A$ is a compact, shiftinvariant subset of $\A^\Z$. 

\begin{definition}
	Let $X$ be a shift space and define an equivalence relation $\sim$ on $X\times \R$ generated by $(x, t+1)\sim (\sigma_X(x), t)$. Giving $X\times \R$ the product topology we let the \emph{suspension flow} of $X$ be given by the quotient space
	\begin{equation*}
		SX = X\times \R/\sim.
	\end{equation*}
	We denote by $[x, t]$ the equivalence class in $SX$ of $(x, t)\in X\times \R$.
\end{definition}
\noindent
A flow equivalence is a homeomorphism between the suspension flows of two shift spaces that preserves direction in $\R$.

\begin{definition}
	Let $X$ and $Y$ be shift spaces and $SX$ and $SY$ their suspension flows. A homeomorphism $\Phi\colon SX\to SY$ is a \emph{flow equivalence} if for each $[x, t]\in SX$ there is a monotonically increasing function $\phi_{[x,t]}\colon \R\to\R$ such that $\Phi([x, t]) = [y, t']$ implies $\Phi([x, t+r])= [y, t'+\phi_{[x, t]}(r)]$.
	
	If such a homeomorphism exists we say that $X$ and $Y$ are \emph{flow equivalent} and write $X\FE Y$.
\end{definition}

\subsection{Flow equivalence as symbol expansions}
To understand flow equivalence in combinatorial terms we need the concept of a symbol expansion. A symbol expansion of a shift space $X$ takes a symbol $a\in \A(X)$ and appends to every occurrence of $a$ in every point of $X$ a symbol $\diamondsuit\not\in \A(X)$ such that $a$ is replaced by $a\diamondsuit$ everywhere in $X$.

\begin{definition}
	Let $X$ be a shift space, $a\in \A(X)$, and $\diamondsuit\not\in \A(X)$. Let for $b\in \A(X)$,
	\begin{equation*}
		\tau(b) = \begin{cases} a\diamondsuit, & b=a\\
					  b, 		     & b\neq a, \end{cases}
	\end{equation*}
and define a function $\mathcal T$ on the points of $X$ by 
	$\mathcal T(x) = \dots \tau(x_{-1}).\tau(x_0)\tau(x_1)\dots$.	
	The shift space $X^{a\mapsto a\diamondsuit} = \mathcal T(X)\cup\sigma(\mathcal T (X))$ is said to be obtained by a \emph{symbol expansion} of $X$.
	If $B$ is a set of words over an alphabet containing $a$, we write 
	\begin{equation*}
		B^{a\mapsto a\diamondsuit} = \{\tau(w_1)\dots \tau(w_n) \mid w_1\dots w_n\in B\}.
	\end{equation*}
\end{definition}
\begin{remark}
Adding the set $\mathcal T(\sigma(X))$ in the definition of the symbol expansion of a shift space is necessary for $X^{a\mapsto a\diamondsuit}$ to be closed under the shift map. We will be rather liberal with the notation $X^{a\mapsto b}$, which will simply mean replacing every occurrence of $a$ in $X$ by $b$ and taking the closure under the shift map.
\end{remark}
\noindent
We now make sure that symbol expansion does in fact yield a shift space.
\begin{proposition}[Johansen \cite{RuneJohansen}]\label{symbolExpansionYieldsShiftSpace}
Let $X$ be a shift space, $a\in \A(X)$, and $\diamondsuit\not\in \A(X)$ be a symbol. Then $X^{a\mapsto a\diamondsuit}$	is a shift space.
\end{proposition}
\begin{proof}
	Let $\F$ be a set of forbidden words for $X$ and $\mathcal B =\A(X)\cup \{\diamondsuit\}$ be an alphabet. The set
	\begin{equation*}
		\F' = \F^{a\mapsto a\diamondsuit} \cup \{b\diamondsuit \mid b\in \A(X)\setminus\{a\}\}\cup \{\diamondsuit\diamondsuit\}
	\end{equation*}
	is a set of forbidden words for $X^{a\mapsto a\diamondsuit} \subseteq \mathcal B^\Z$.
\end{proof}
\noindent
A result by Parry and Sullivan makes the concept of flow equivalence available to our combinatorial interpretation of the theory of shift spaces through symbol expansion. They showed that flow equivalence is the coarsest equivalence relation which is closed under both conjugacy and symbol expansion. 
\begin{theorem}[Parry and Sullivan \cite{ParrySullivan}]\label{PaSu}
	Let $X, Y$ be shift spaces. Then $X\FE Y$ if and only if there exists a sequence of shift spaces $X_0=X, X_1, \dots, X_n=Y$ such that for each $0\leq i < n$ one of the following conditions hold.
	\begin{itemize}
		\item $X_{i}$ is obtained by a symbol expansion of $X_{i+1}$.
		\item $X_{i+1}$ is obtained by a symbol expansion of $X_i$.
		\item $X_i$ and $X_{i+1}$ are conjugate.
	\end{itemize}
\end{theorem}

\section{Basic results on flow equivalence}
In this section, we develop the basic tools needed in later chapters for dealing with flow equivalence. 

\noindent
First, symbol expansions of a shift space $X$ using symbols of $\A(X)$ are possible.

\begin{proposition}[Johansen \cite{RuneJohansen}]\label{extendWithExistingSymbol}
Let $X$ be a shift space and $a, b\in \A(X)$ with $a\neq b$. Then $X\FE X^{a\mapsto ab}$.	
\end{proposition}
\begin{proof}
	Let $\diamondsuit\not\in \A(X)$. Then $X\FE X^{a\mapsto a\diamondsuit}$ by Theorem \ref{PaSu}.
	
	Now, consider $\Phi\colon B_1(X\ex{a}{a\diamondsuit})\to \A(X)$ given by 
	\begin{equation*}
		\Phi(w_1) = \begin{cases}
			w_1, & w_1 \in\A(X)\\
			b, & w_1=\diamondsuit,
		\end{cases}
	\end{equation*}
	which induces a sliding block code $\phi=\Phi_\infty^{[0,0]}\colon X\ex{a}{a\diamondsuit} \to X\ex{a}{ab}$. As the sliding block code $\psi = \Psi_\infty^{[1, 0]}$ induced by the 2-block code $\Psi \colon B_2(X\ex{a}{ab})\to \A(X)\cup \{\diamondsuit\}$ satisfying
	\begin{equation*}
		\Psi(w_1w_2) = \begin{cases}
			\diamondsuit, & w_1w_2=ab \\
			w_2, & w_1w_2\neq ab
		\end{cases}
	\end{equation*}
	is an inverse of $\phi$, it is a conjugacy, and $X\FE X\ex{a}{a\diamondsuit}\FE X\ex{a}{ab}$.
\end{proof}
\noindent
Second, we can remove symbols from a shift space by a reversal of the process of symbol expansion, a \emph{symbol contraction}.

\begin{lemma}[Johansen \cite{RuneJohansen}]\label{collapseTwoAlwaysConsecutiveLetters}
Let $X$ be a shift space and $a, b\in \A(X)$ with $a\neq b$. If for every point $x\in X$ it is true that $x_{i+1}=b$ if and only if $x_i=a$, then $X\FE X\ex{b}{\epsilon}$.	
\end{lemma}
\begin{proof}
	First of all, let $Y=X\ex{b}{\epsilon}$ and a set of forbidden words for $X$ be given by
	\begin{equation*}
		\F = \F' \cup \{ cb\mid c\in \A(X)\setminus \{a\} \}, 
	\end{equation*}
	where every occurrence of $b$ in $\F'$ is preceded by an $a$ and every occurrence of $a$ is followed by a $b$.
	Then $Y\subseteq (\A(X)\setminus \{b\})^\Z$ is a shift space with forbidden words $\F^* = {\F'} \ex{b}{\epsilon}$. Note that none of the points of $X$ ``vanish'' under the removal of $b$, since every $b$ is preceded by an $a$.
	
	 Now, $X=Y\ex{a}{ab}$ so by Proposition \ref{extendWithExistingSymbol}, $X\FE Y$.
\end{proof}

\begin{lemma}[Johansen \cite{RuneJohansen}]\label{standardSymbolContraction}
	Let $X$ be a shift space and $a, b\in \A(X)$ be different. If for every point $x\in X$ we have $x_{i+1}=b$ if $x_i=a$, then $X\FE X\ex{ab}{a}$.
\end{lemma}
\begin{proof}
	Let $\diamondsuit\not\in \A(X)$ and define a 2-block map
	\begin{equation*}
		\Psi(w_1w_2) = \begin{cases}
			\diamondsuit, & w_1w_2=ab\\
			w_2, & w_1w_2\neq ab.
		\end{cases}
	\end{equation*}
	As we saw in the proof of Proposition \ref{extendWithExistingSymbol} the sliding block code $\psi = \Psi_\infty^{[1,0]}$ is a conjugacy. The image $Y=\psi(X)$ is a shift space where for all $y\in Y$, $y_i=a$ if and only if $y_{i+1}=\diamondsuit$, so by Lemma \ref{collapseTwoAlwaysConsecutiveLetters}, 
$X\ex{ab}{a} = Y\ex{\diamondsuit}{\epsilon} \FE Y\cong X$.		
\end{proof}

\noindent
Third, we can contract words into single symbols. 

For words $u, w$ we say that $u$ is a \emph{prefix} of $w$ if we can write $w=uv$ for some word $v$. The term \emph{suffix} is defined analogously.

\begin{definition}
	Let $X$ be a shift space and $w\in B_n(X)$. The space $X$ is said to admit \emph{non-trivial $w$-overlaps} if there is a  $v\in B_m(X), n< m<2n$ with prefix and suffix $w$.
\end{definition}

\begin{definition}
	A word $w$ over some alphabet $\mathcal A$ is said to be \emph{non-overlapping}, if the shift space $\mathcal A^\Z$ does not allow non-trivial $w$-overlaps, i.e., if no proper prefix of $w$ is also a suffix of $w$.
\end{definition}

\begin{remark}
	If $w$ is non-overlapping, it follows that no shift space $X$ with $w\in B(X)$ allows non-trivial $w$-overlaps.
\end{remark}

\begin{proposition}[Johansen \cite{RuneJohansen}] \label{replOverlap} For any shift space $X$, any $w\in B(X)$ for which $X$ does not admit non-trivial $w$-overlaps, and $\diamondsuit\not\in\mathcal A(X)$ we have $X\sim_{FE} X^{w\mapsto \diamondsuit }$.
\end{proposition}
\begin{proof}
	Let $w=w_1\dots w_n$ and define an $n$-block map
	\begin{equation*}
		\Phi(e_1\dots e_n) = \begin{cases}
			\diamondsuit, & e_1\dots e_n = w\\
			e_1, & e_1\dots e_n\neq w.
		\end{cases}
	\end{equation*}
	Then $\phi = \Phi_\infty^{[0, n-1]} \colon X\to X^{w\mapsto \diamondsuit w_2\dots w_n}$ is a conjugacy the inverse of which simply replaces $\diamondsuit$ with $w_1$ everywhere. Note that this is only well defined as $X$ does not allow non-trivial $w$-overlaps. Now, $w_2$ always follows $\diamondsuit$ in $X\ex{w}{\diamondsuit w_2\dots w_n}$, so by Lemma \ref{standardSymbolContraction}, $X\ex{w}{\diamondsuit w_2\dots w_n}\FE X\ex{w}{\diamondsuit w_3\dots w_n}$. Repeating this procedure, removing a $w_i$ at a time, yields
	\begin{equation*}
		X\cong X\ex{w}{\diamondsuit w_2\dots w_n}\FE X\ex{w}{\diamondsuit w_3\dots w_n} \FE \dots \FE X\ex{w}{\diamondsuit}.
	\end{equation*}
\end{proof}
\noindent
Fourth, shift spaces being of finite type, irreducible or sofic is preserved under flow equivalence, though we omit the proof for brevity.
\begin{proposition}[Johansen \cite{RuneJohansen}]
If the shift space $X$ is of finite type, irreducible, or sofic, and $Y\FE X$, then $Y$ is of finite type, irreducible, or sofic, respectively.
\end{proposition}

\section{Classification of shifts of finite type}
Arguably, the most important classification result on flow equivalence is due to John Franks \cite{Franks} who completely classified the irreducible shifts of finite type. His paper from 1984 is based on works by Bowen-Franks and Parry-Sullivan and presents a complete and easily computable invariant of flow equivalence on the matrices of edge shifts. We will introduce the invariant and show its necessity. The proof of the sufficiency of the invariant, which is the content of Franks' paper from 1984, is mostly based on matrix manipulation and will not be detailed here.

The proofs of this section are inspired by material from Lind and Marcus \cite{LM}, Parry and sullivan \cite{ParrySullivan}, and Franks \cite{Franks}.
\subsection{The signed Bowen-Franks group}
The invariant introduced by Franks is the signed Bowen-Franks group, which is an augmentation of the Bowen-Franks group.

\begin{definition}
	Let $A$ be an $n\times n$ integer matrix with non-negative entries. The \emph{Bowen-Franks group} of $A$ is given by the quotient
	\begin{equation*}
		\BF(A) = \Z^n/(I-A)\Z^n.
	\end{equation*}
\end{definition}
\noindent
The Bowen-Franks group can be computed quite easily with elementary linear algebra. For the square integer matrix $A$, we consider the \emph{elementary operations over $\Z$} on $A$ to be the following:
\begin{enumerate}
	\item Exchanging two rows or columns of $A$.
	\item Multiplying a row or  column of $A$ by $-1$.
	\item Adding an integer multiple of a row or column of $A$ to another row or column, respectively, of $A$.
\end{enumerate}
Every such elementary operation over $\Z$ has an \emph{elementary matrix} representing it. An elementary matrix is a square integer matrix, which performs an elementary operation on a matrix if multiplied from either the left or the right. Since every elementary operation has an inverse elementary operation, all elementary matrices are invertible over $\Z$.

We recall that the Smith normal form of $A$ is the unique diagonal matrix, reachable from $A$ by elementary operations over $\Z$, of the form
\begin{equation}\label{smith}
	D = \begin{pmatrix}
		d_1&\cdots&0\\
		 \vdots&\ddots &\vdots\\
		0&\cdots&d_n
	\end{pmatrix},
\end{equation}
	where $d_i\geq 0$ and $d_i \mid d_{i+1}$ for all $i$ with the convention that every integer divides zero and no positive integer divides zero. It turns out that we can derive the Bowen-Franks group of a matrix from its  Smith normal form.

\begin{lemma}\label{isomorphism}
	Let $B$ be an $n\times n$ integer matrix and $E$ an elementary matrix. Then
	\begin{equation*}
		\Z^n/B\Z^n \simeq \Z^n/(BE\Z^n) \simeq \Z^n/(EB\Z^n).
	\end{equation*}
\end{lemma}
\begin{proof}
	Since $E$ is invertible over $\Z$, $E\Z^n = \Z^n$, so $\Z^n/(B\Z^n) = \Z^n/(BE\Z^n)$. Further, we have an isomorphism $\varphi\colon \Z^n \to \Z^n$ given by $\varphi(v) = Ev$, which yields
	\begin{equation*}
		\Z^n/(EB\Z^n) = \varphi(\Z^n)/\varphi(B\Z^n) \simeq \Z^n/B\Z^n.
	\end{equation*}
\end{proof}

\begin{proposition}
If $I-A$ has Smith normal form $D$ written as \eqref{smith}, then 
\begin{equation*}
	\BF(A) \simeq \Z_{d_1}\oplus \dots \oplus \Z_{d_n},
\end{equation*}
where $\Z_0=\Z$ and $\Z_n = \Z/n\Z, n\neq 0$.
\end{proposition}
\begin{proof}
	Since $D$ is reachable from $I-A$ by elementary operations, Lemma \ref{isomorphism} yields
	\begin{equation*}
		\Z^n/(I-A)\Z^n\simeq \Z^n/D\Z^n.
	\end{equation*}
	The proposition now follows, as $D\Z^n = \{ (a_1 d_1, \dots, a_n d_n)\in \Z^n \mid a_i\in \Z \}$.
\end{proof}

\noindent
Sadly, the Bowen-Franks group does not constitute a complete invariant for flow equivalence, so we introduce an additional component.

\begin{definition}
	The \emph{signed Bowen-Franks group} is given by the pair
	\begin{equation*}
		\BF_+(A) = (\sign \det(I-A), \BF(A)).
	\end{equation*}
	We write $\BF_+(A)\simeq \BF_+(B)$ if $\sign \det(I-A)=\sign \det(I-B)$ and $\BF(A)\simeq \BF(B)$.
\end{definition}
\noindent
Actually, the complete invariant is given by the determinant $\det(I-A)$ and the Bowen-Franks group $\BF(A)$, but $\abs{\det(I-A)}$ can be extracted from $\BF(A)$ since $\abs{\det(I-A)}=\abs{\det(D)}$, where $D$ is the Smith normal form of $I-A$. Thus, only the sign of the determinant is necessary for the complete invariant. 

\subsection{Franks' result}
We say that the matrix $A$ is irreducible if the shift space $\X_A$ is irreducible. The result by Franks now reads.
\begin{theorem}[Franks \cite{Franks}]
	Suppose that $A$ and $B$ are non-negative irreducible integer matrices such that neither $\X_A$ nor $\X_B$ is a single orbit. Then $\X_A\FE \X_B$ if and only if $\BF_+(A) \simeq \BF_+(B)$.
\end{theorem}

\begin{example}\label{fullShiftsAreDifferent}
	Let $r>1$ be an integer. The full $r$-shift $\X_{[r]}$ has matrix representation $A_r = \left(r\right)$, we have $\det(I-A_r) = 1-r$, and $(r-1)$ is the Smith normal form of $I-A_r$. So the signed Bowen-Franks group of $A_r$ is
	\begin{equation*}
		\BF_+(A_r) \simeq (-, \Z_{r-1}).
	\end{equation*}
	
	It follows that for two different integers $r,s>1$, $\X_{[r]}\not\FE \X_{[s]}$. Thus, two full shifts are flow equivalent if and only if they are conjugate.
\end{example}
\noindent
As the above example illustrates, the signed Bowen-Franks group is a very convenient invariant as it is easily computed given two matrices. We will now show the necessity of the invariant.
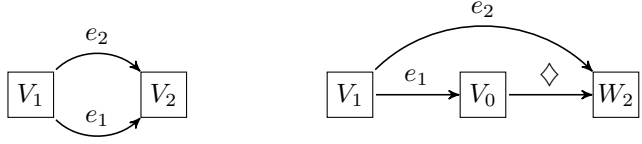
\begin{figure}
\begin{center}
\begin{tikzpicture}
  [bend angle=45,
   knude/.style = {circle, inner sep = 0pt},
   vertex/.style = {circle, draw, minimum size = 1 mm, inner sep =
      0pt,},
   textVertex/.style = {rectangle, draw, minimum size = 6 mm, inner sep =
      1pt},
   to/.style = {->, shorten <= 1 pt, >=stealth', semithick}, scale=0.7]
  
  \node[knude] (trans) at (6,0) {} ;
  
  \node[textVertex] (V1) at (0,0) {$V_1$};
  \node[textVertex] (V2) at (2.5,0) {$V_2$};

  \node[textVertex] (P1) at ($(V1)+(trans)$) {$V_1$};
  \node[textVertex] (P2) at ($(2.5,0)+(V1)+(trans)$) {$V_0$};
  \node[textVertex] (W1) at ($(5,0)+(V1)+(trans)$) {$W_2$};

  \draw[to, bend right=45] (V1) to node[auto] {$e_1$} (V2);
  \draw[to, bend left=45] (V1) to node[auto] {$e_2$} (V2);

  \draw[to] (P1) to node[auto] {$e_1$} (P2);
  \draw[to] (P2) to node[auto] {$\diamondsuit$} (W1);
  \draw[to, bend left = 45] (P1) to node[auto] {$e_2$} (W1);

\end{tikzpicture}
\caption{Illustration of a symbol expansion of a shift of finite type.}\label{figure:expansion}
\end{center}
\vspace{-0.5cm}
\end{figure}
\begin{lemma}\label{matrixReprOfExpansion}
	For a non-negative integer matrix $A=(a_{ij})$ with $a_{kl}>0$ we define
	\begin{equation*}
		\bar{A} = \begin{pmatrix}
		  			  	0&0 &\cdots&1&\cdots & 0\\
		  			  	0& a_{11} & \cdots & a_{1l}&\cdots & a_{1n}\\
		  			  	\vdots & \vdots & & \vdots && \vdots \\
		  			  	1 & a_{k1} & \cdots & a_{kl}-1 &\cdots & a_{kn}\\
		  			  	\vdots & \vdots & & \vdots && \vdots \\ 
		  			  	0& a_{n1} & \cdots & a_{nl}&\cdots & a_{nn}\\
					  \end{pmatrix}.
	\end{equation*}
	Write $X\sim_{\text{SE}} Y$ if $X=\bar Y$ or $Y=\bar X$. Flow equivalence is generated by the relations $\sim_{\text{SE}}$ and $\simequiv$.
\end{lemma}
\begin{proof}
	As Figure \ref{figure:expansion} illustrates, the edge shift $\X_{\bar A}$ is a symbol expansion of $\X_A$ and every symbol expansion can be represented in that way. Further, strong shift equivalence is generated by elementary equivalences, so the result follows as flow equivalence is generated by strong shift equivalence and symbol expansion.
\end{proof}
\begin{proposition}
	If $A, B$ are non-negative integer matrices with $\X_A \FE \X_B$, then $\BF(A)\simeq \BF(B)$.
\end{proposition}
\begin{proof}
	We show separately that the Bowen-Franks group is invariant up to isomorphism for elementary equivalences and for symbol expansions. As seen in Lemma \ref{matrixReprOfExpansion} these relations generate flow equivalence so the result then follows.
	
	Let $A$ and $B$ be two elementary equivalent integer matrices with $A=RS$ and $B=SR$ for non-negative integer matrices $R, S$, and let $m$ and $n$ be the size of $A$ and $B$, respectively. Then for any $v\in \Z^m$ we have
	\begin{equation*}
		S(I_m-A)v = Sv-SRSv = Sv-BSv = (I_n-B)Sv,
	\end{equation*}
	so $S(I_m-A)\Z^m \subseteq (I_n-B)\Z^n$. Hence, the map $\hat S\colon \BF(A)\to \BF(B)$ given by 
	\begin{equation*}
		\hat S(v+ (I_m-A)\Z^m) ) = Sv + (I_n-B)\Z^n
	\end{equation*}
	for $v\in \Z^m$ is a well-defined group homomorphism. Similarly, we have a group homomorphism $\hat R\colon \BF(B)\to \BF(A)$ given by 
	$\hat R(v+(I_n-B)\Z^n) = Rv+(I_m-A)\Z^m$
	 for $v\in \Z^n$. The homomorphism $\hat A = \hat R \circ \hat S $ given by 
	 $\hat A(v+(I_m-A)\Z^m) = Av + (I_m-A)\Z^m$
	  is now the identity, since for every $v\in \Z^m$,
	\begin{align*} 
		\hat A (v+(I_m-A)\Z^m)-(v+(I_m-A)\Z^m) &= (A-I_m)v + (I_m-A)\Z^m \\&= (I_m-A)\Z^m.
	\end{align*}
	Similarly, $\hat B = \hat S \circ \hat R$ is the identity on $\BF(B)$, so $\hat R$ and $\hat S$ are isomorphisms and we get $\BF(A)\simeq \BF(B)$.
	
	For symbol expansion, we see that we can alter $I-\bar A$ by one row and one column operations to
	\begin{align*}
	I-\bar{A} &=\begin{pmatrix}
		  			  	1&0 &\cdots&-1&\cdots & 0\\
		  			  	0& 1-a_{11} & \cdots & -a_{1l}&\cdots & -a_{1n}\\
		  			  	\vdots &  \vdots& & \vdots && \vdots \\
		  			  	-1 & -a_{k1} & \cdots & -a_{kl}+1 &\cdots & -a_{kn}\\
		  			  	\vdots & \vdots & & \vdots && \vdots \\ 
		  			  	0& -a_{n1} & \cdots & -a_{nl}&\cdots & 1-a_{nn}\\
					  \end{pmatrix}\\
		&\to \begin{pmatrix}
			1 & 0&\cdots &0\\
			0 & 1-a_{11}&\cdots & -a_{1n}\\
			\vdots &\vdots &\ddots&\vdots\\
			0&-a_{n1}&\cdots & 1-a_{nn}
		\end{pmatrix}.
	\end{align*}
	Note that if $k=l$, then $I-\bar A$ is written differently, although the result remains. So if $D$ is the Smith normal form of $I-A$ written as in \eqref{smith}, then the Smith normal form of $I-\bar A$ is 
	\begin{equation*}
		\begin{pmatrix}
			1 &0 &\cdots &0\\
			 0& d_1&\cdots &0\\
			 \vdots&\vdots &\ddots&\vdots\\
			0&0&\cdots& d_n
		\end{pmatrix}.
	\end{equation*}
	This yields
	\begin{equation*}
		\BF(\bar A) \simeq \Z_1 \oplus \Z_{d_1}\oplus \dots \oplus \Z_{d_n} \simeq \Z_{d_1}\oplus \dots \oplus \Z_{d_n} \simeq \BF(A).
	\end{equation*}
\end{proof}

\begin{lemma}[Sylvester's theorem] Let $A$ be an $m\times n$ and $B$ an $n\times m$ matrix. Then
	$\det(I_m + AB) = \det(I_n+BA)$.
\begin{proof}
	Define the block matrix
\begin{equation*}
	M = \begin{pmatrix}
		    I_m & -A\\
		    B   & I_n
		\end{pmatrix}
\end{equation*}
and see that we can decompose $M$ as 
\begin{equation*}
	M = \begin{pmatrix}
		    I_m & 0\\
		    B   & I_n
		\end{pmatrix}
		\begin{pmatrix}
		    I_m & -A\\
		    0   & I_n+BA
		\end{pmatrix}\,\, \text{ and }\,\,
		M = \begin{pmatrix}
		    I_m+AB & -A\\
		    0   & I_n
		\end{pmatrix}
		\begin{pmatrix}
		    I_m & 0\\
		    B   & I_n
		\end{pmatrix}.
\end{equation*}
Taking determinants we get $\det(I_n+BA) = \det(M) = \det(I_m+AB)$.
\end{proof}
\end{lemma}
\begin{proposition}
	If $A, B$ are non-negative integer matrices with $\X_A \FE \X_B$ then $\det(I-A)=\det(I-B)$.
\end{proposition}
\begin{proof}
	We show invariance separately for elementary equivalences and for symbol expansions. As seen in Lemma \ref{matrixReprOfExpansion} these relations generate flow equivalence and the result follows.
	
	Let $C$ and $D$ be elementary equivalent non-negative integer matrices with $C=RS$ and $D=RS$ for non-negative integer matrices $S, R$. Then it follows by Sylvester's theorem that
	\begin{equation*}
		\det(I-C) = \det(I+(-R)S) = \det(I+S(-R)) = \det(I-D).
	\end{equation*}
	
	For symbol expansion, let $A$ be an integer matrix. Then by adding rows, we get
	\begin{align*}
		\det (I-\bar A) &= \det\begin{pmatrix}
		  			  	1&0 &\cdots&-1&\cdots & 0\\
		  			  	0& 1-a_{11} & \cdots & -a_{1l}&\cdots & -a_{1n}\\
		  			  	\vdots &  \vdots& & \vdots && \vdots \\
		  			  	-1 & -a_{k1} & \cdots & -a_{kl}+1 &\cdots & -a_{kn}\\
		  			  	\vdots & \vdots & & \vdots && \vdots \\ 
		  			  	0& -a_{n1} & \cdots & -a_{nl}&\cdots & 1-a_{nn}\\
					  \end{pmatrix}\\
		                &=  \det\begin{pmatrix}
		  			  	1&0 &\cdots&-1&\cdots & 0\\
		  			  	0& 1-a_{11} & \cdots & -a_{1l}&\cdots & -a_{1n}\\
		  			  	\vdots &  \vdots& & \vdots && \vdots \\
		  			  	0 & -a_{k1} & \cdots & -a_{kl} &\cdots & -a_{kn}\\
		  			  	\vdots & \vdots & & \vdots && \vdots \\ 
		  			  	0& -a_{n1} & \cdots & -a_{nl}&\cdots & 1-a_{nn}\\
					  \end{pmatrix}\\
		&= \det\begin{pmatrix}
			1-a_{11}&\cdots & -a_{1n}\\
			\vdots  &\ddots&\vdots &\\
			-a_{n1}&\cdots & 1-a_{nn}
		\end{pmatrix}\\
		&= \det(I-A)
	\end{align*} 
	where the last equality follows by expansion of the $l$th column. Note, again, that the above representations of matrices are not accurate when $k=l$, but the calculation is still the same.
\end{proof}

\chapter{Entropy of flow classes}
This chapter will study the entropy of shift spaces under flow equivalence. As in physics or information theory, entropy is an indicator of the number of allowed states or words of a shift space, the more states the greater the entropy. After developing the necessary theory, we show some new results that shed light on how wildly entropy can change under flow equivalence.

\begin{remark}
Any result in this chapter that does not have a citation or is stated as known, is new and produced by the author. Further, as no proof of Theorem \ref{zeroEntropyInvariant} could be found in the literature, the one presented here has been supplied by the author.
\end{remark}

\section{Introduction to entropy}
Entropy describes the information density or complexity of a shift space by the asymptotic number of words of a given length. 
\begin{definition}
	Let $X$ be a shift space. Then the \emph{entropy} of $X$ is given by
	\begin{equation*}
		h(X) = \lim_{n\to\infty}\frac{1}{n}\log \abs{B_n(X)},
	\end{equation*}
	where $\log$ is the base 2 logarithm.
\end{definition}
\noindent
The limit exists (see for instance Lind and Marcus \cite[Prop. 4.1.8]{LM}), so the entropy is always well-defined. Entropy can be said to describe the information density of a shift space in the sense that if $h(X)=t>0$ for some shift space $X$, then there are roughly $2^{tn}$ words of length $n$ in $X$. A very intuitive example of entropy is that of the full shift.
\begin{example}\label{fullREntropy}
Let $X=X_{[r]}$ be the full $r$-shift. Then $\abs{B_n(X)}=r^n$, so 
\begin{equation*}
	h(X) = \lim_{n\to\infty}\frac{1}{n}\log\abs{B_n(X)} = \lim_{n\to\infty}\frac{n}{n}\log r = \log r.
\end{equation*}
\end{example}

\subsection{Entropy under sliding block codes}
Entropy turns out to be invariant under conjugacy, and in general sliding block codes between shift spaces behave nicely with respect to entropy, which will be useful in later sections.
\begin{proposition}[Lind and Marcus \cite{LM}]\label{factorEntropy}
Let $Y$ be a factor of the shift space $X$. Then $h(Y)\leq h(X)$.
\end{proposition}
\begin{proof}
	Since $Y$ is a factor of $X$ there is a surjective sliding block code $\phi:X\to Y$ induced by some $(M+N+1)$-block map $\Phi:B_{M+N+1}(X)\to \A(Y)$. Thus, for every $w\in B_n(y)$ there must exist $u\in B_{M+N+n}(X)$ such that $\Phi(u) = w$. 
	So $\abs{B_n(Y)}\leq \abs{B_{M+N+n}(X)}$ and hence
	\begin{align*}
		h(Y) &= \lim_{n\to\infty}\frac{1}{n}\log\abs{B_n(Y)}\\
		     & \leq \lim_{n\to\infty}\frac{1}{n}\log\abs{B_{M+N+n(X)}}\\
		     & = \lim_{n\to\infty}\frac{M+N+n}{n}\left(\frac{1}{M+N+n}\log\abs{B_{M+N+n(X)}}\right)\\
		     & = h(X).
	\end{align*}
\end{proof}
\noindent
The invariance of entropy under conjugacy now follows almost trivially.
\begin{proposition}[Lind and Marcus \cite{LM}]\label{entroConjInv}
Let $X$ and $Y$ be conjugate shift spaces. Then $h(X)=h(Y)$.	
\end{proposition}
\begin{proof}
	Let $\phi\colon X\to Y$ be a conjugacy. Since $\phi$ is bijective, $X$ is a factor of $Y$ and $Y$ is a factor of $X$, so $h(X)=h(Y)$ by Proposition \ref{factorEntropy}.
\end{proof}
\noindent
A last result that will be crucial later is rather intuitive given the above propositions.

\begin{proposition}[Lind and Marcus \cite{LM}] \label{entrOfFull}
Let the shift space $X$ embed into $Y$. Then $h(X)\leq h(Y)$.	
\end{proposition}
\begin{proof}
	There is an embedding $\phi\colon X\to Y$ and its image $\phi(X)\subseteq Y$ is a shift space by Proposition \ref{imageUnderBlockCode}. Since $\phi$ is injective it is a conjugacy from $X$ to $\phi(X)$ by Proposition \ref{bijHasInverse}, so $h(X)= h(\phi(X))$ by Proposition \ref{entroConjInv}. Now, for all $n\in \N$ we have $\abs{B_n(\phi(X))}\leq \abs{B_n(Y)}$  as $\phi(X)\subseteq Y$, and thus, $h(X)=h(\phi(X))\leq h(Y)$.
\end{proof}

\subsection{Entropy of sofic shifts}
Since entropy measures exponential growth, it seems natural for a shift space $X$ that subshifts of $X$ with an entropy less than $X$ can be disregarded in calculating the entropy of $X$ as they only contribute subexponentially. This is the basis of the next proposition whose proof is omitted since the necessary theory is not developed in this presentation. 
\begin{proposition}[Lind and Marcus \cite{LM}]\label{irreducibleEntropy}
Let $G$ be a graph and let $T$ be the set of irreducible subgraphs of $G$. Then 
\begin{equation*}
	h(\X_G) = \max_{H\in T}h(\X_H).
\end{equation*}
\end{proposition}
\begin{proof}
	See Lind and Marcus \cite[Theorem 4.3.1 and 4.4.4]{LM}.
\end{proof}
\noindent
In order to say something about the entropy of sofic shifts later, we need the following two propositions, which connect the presentation of a sofic shift as a labeled graph with its irreducibility and entropy.
\begin{proposition}
Let $X$ be sofic with representation $\G=(G, \Lab)$. If $G$ is irreducible then $X$ is irreducible.	
\end{proposition}
\begin{proof}
	Similar to the first part of the proof of Proposition \ref{irreducibleMatrixVsGraph}.
\end{proof}
\begin{proposition}[Lind and Marcus \cite{LM}]\label{soficEntropyByGraph}
Let $\mathcal G = (G, \mathcal L)$ be a right-resolving labeled graph. Then $h(\X_{\mathcal G}) = h(\X_G)$.	
\end{proposition}
\begin{proof}
Let $k$ be the number of vertices of $G$ and $w\in B_n(\X_\G)$. There is at least one and at most $k$ paths $\pi$ on $G$ with $\Lab(\pi)=w$, because $G$ being right-resolving implies that each vertex can support at most one such path. This means that $B(\X_G)\geq B(\X_\G)\geq \frac{1}{k}B(\X_G)$, yielding 
	\begin{equation*}
		h(\X_G) = \lim_{n\to\infty} \frac{1}{n} \log B_n(\X_G) \geq h(\X_\G) \geq \lim_{n\to\infty} \frac{1}{n} \log \left(\tfrac1kB_n(X_G)\right) = h(\X_G).
	\end{equation*}
\end{proof}

\subsection{Terminology for remaining chapter}
The remainder of this chapter will be devoted to the study of the entropy of the flow equivalence classes of shift spaces.

\begin{definition}
	Let $X$ be a shift space. We denote by $[X]$ the class of shift spaces that are flow equivalent to $X$, and by $h([X])$ the set $\{h(Y)\mid Y\in [X]\}$.
\end{definition}
\noindent
In particular we shall be interested in whether or not $h([X])$ is bounded from above or below in $\R^+$.

\begin{definition}
	We say that a shift space $X$ is \emph{flow equivalent to shifts of arbitrarily large entropy} if $h([X])$ is unbounded, and similarly we say that $X$ is \emph{flow equivalent to shifts of arbitrarily small entropy} if $\inf (h([X])) = 0$.
\end{definition}

\section{Small entropies}
This section will show that all shift spaces are flow equivalent to shifts of arbitrarily small entropy. First of all, we prove the already known result  that having entropy zero is an invariant under flow equivalence.
\begin{theorem}\label{zeroEntropyInvariant}
	Let $X$ and $Y$ be shift spaces with $X\FE Y$. Then $h(X)=0$ if and only if $h(Y)=0$.
\end{theorem}
\begin{proof}
	Let $X$ be a shift space. Since entropy is invariant under conjugacy and flow equivalence is generated by conjugacy and symbol expansion, we only need to show that for some shift space $X$ and some shift space $Y=X^{a\mapsto a\diamondsuit}$ obtained by a symbol expansion of $X$, we have $h(X)=0$ if and only if $h(Y)=0$. 
	
	First, for $u_1, u_2\in B_n(Y)$, it holds that $u_1\ex{\diamondsuit}{\epsilon}$ is a prefix of $u_2\ex{\diamondsuit}{\epsilon}$ if and only if $u_2$ can be achieved by adding and removing $\diamondsuit$ at the ends of $u_1$. This can be done in maximally of two different ways (e.g. if $u_1=\diamondsuit u_1'$, where $u_1'$ does not end in a $\diamondsuit$, then $u_2=u_1$ or $u_2=u_1'\diamondsuit$ are the two possibilities), so for every $w\in B_n(X)$ there is at most two words $u \in B_n(Y)$ such that $u\ex{\diamondsuit}{\epsilon}$ is a prefix of $w$, and for every $u \in B_n(Y)$, $u\ex{\diamondsuit}{\epsilon}$ is a prefix of some $w\in B_n(X)$. Thus, $2\abs{B_n(X)}\geq \abs{B_n(Y)}$, which yields
	\begin{equation*}
		h(X) = \lim_{n\to\infty}\frac{1}{n}\log 2\abs{B_n(X)} \geq \lim_{n\to\infty}\frac{1}{n}\log \abs{B_n(Y)} = h(Y).
	\end{equation*}
	
	Second, for two different word $w_1, w_2\in B_n(X)$, none of the words $w_1\ex{a}{a\diamondsuit}$ and $w_2\ex{a}{a\diamondsuit}$ can be a prefix of the other, and for every $w\in B_n(X)$ there is at least one word $u\in B_{2n}(Y)$, which has $w\ex{a}{a\diamondsuit}$ as a prefix. So $B_{2n}(Y)\geq B_n(X)$, and we can make the estimate

	\begin{align*}
		h(Y) = \lim_{n\to \infty} \frac{1}{2n}\log \vert B_{2n}(Y)\vert 
		      \geq  \lim_{n\to \infty} \frac{1}{2n}\log \vert B_{n}(X)\vert
		     = \frac{1}{2}h(X).
	\end{align*}
	Thus, $h(X)\geq h(Y)\geq \frac{1}{2}h(X)$ and the result follows.
\end{proof}
\noindent
Moving on to shift spaces with non-zero entropy, we need a procedure that given a shift space can produce shift spaces flow equivalent to it of arbitrarily small entropy.
\begin{theorem}\label{fractionalEntropy}
	Let $X$ be a shift space and $n\in \N$. Then there exists $Y\FE X$ with $h(Y)=\frac{1}{n}h(X)$.
\end{theorem}
\begin{proof}
	The case $n=1$ is trivially true, so assume that $n>1$. Let $\mathcal A=\{e_1, e_2, \dots, e_m\}$ be the alphabet of $X$, $\diamondsuit \not\in\mathcal A$, and $w=\diamondsuit^{n-1}$. Further, set $X_0 = X$ and consider the series of symbol expansions 
	\begin{equation*}
		X_i=X_{i-1}^{e_{i}\mapsto e_{i}w}, 1\leq i\leq m.
	\end{equation*}
	 Now,  $X\FE X_m$ by repeated use of Proposition \ref{extendWithExistingSymbol}, and for every $s\in \N$ the words of $X_m$ of length $ns$ can be described by
	\begin{equation*}
		B_{ns}(X_m) = \{\diamondsuit^k f_1 w f_2w\dots w f_s\diamondsuit^{n-1-k}\mid 0\leq k\leq n-1\text{ and }f_1f_2\dots f_s\in B_s(X)\}.
	\end{equation*}
 So, noting that $\abs{B_{ns}(X_m)} = n\abs{B_s(X)}$, we find that
	\begin{align*}
		\frac 1n h(X) &= \frac1n \lim_{s\to \infty} \frac 1s \log \vert B_s(X) \vert
		       = \lim_{s\to \infty} \frac 1{ns} \log \left(\tfrac1n \vert B_{ns}(X_m)\vert\right)
		       = h(X_m).
	\end{align*}
\end{proof}
\noindent
The main result of the section now follows easily.
\begin{corollary}
	Any shift space $X$ is flow equivalent to shifts of arbitrarily small entropy.
\end{corollary}
\begin{proof}
Follows directly from Theorem \ref{fractionalEntropy}.
\end{proof}

\section{Large entropies}
Let $\R^+$ denote the strictly positive real numbers. In this final section of the chapter we show that shift spaces with non-zero entropy from a wide range of classes are flow equivalent to shifts of arbitrarily large entropy and that this is in fact equivalent to the entropy under flow equivalence being dense in $\R^+$. The author is not at the moment aware of any shift space that is a candidate for having a bounded non-zero entropy.

\subsection{General results}
Before we consider specific classes of shift spaces, we need some preliminary results. 

\begin{theorem}\label{dense}
	Let $X$ be a shift space, which is flow equivalent to shifts of arbitrarily large entropy. Then $h([X])$ is dense in $\R^+$.
\end{theorem}
\begin{proof}
	Let $i\in \R^+$ be arbitrary and $\epsilon>0$ be given. From the assumptions we can find a $Y\FE X$ such that $T=h(Y)$ satisfies $0<\frac{i^2}{T-i}<\epsilon$. 
	
	Let $n\in \N$ be given such that $\frac T{n+1} \leq i\leq \frac{T}{n}$
	or equivalently $\frac Ti-1\leq n\leq \frac Ti$. Then Theorem \ref{fractionalEntropy} allows us to find $Z\FE Y\FE X$ with $h(Z)=\frac{1}{n}h(Y)$ and
	\begin{equation*}
	    \left\vert h(Z)-i\right\vert =\left\vert\frac Tn-i\right\vert \leq \frac{T}{n}-\frac{T}{n+1} = \frac{T}{n(n+1)}\leq \frac{T}{\left(\frac Ti-1\right)\frac Ti}=\frac{i^2}{T-i}<\epsilon.
	\end{equation*}
\end{proof}
\noindent 
One idea, for proving a shift space $X$ to be flow equivalent to shifts of arbitrarily large entropy, is to find a copy of the full 2-shift embedded in a shift $Y\FE X$ and then use it to construct spaces of increasingly large entropy.

Once the shift $Y$ is found the procedure is pretty straight-forward. 
\begin{proposition} \label{2shiftEmb}
	Let $X$ be a shift space. If there is an injective 1-block code $\phi:\X_{[2]}\to X$, then $h([X])$ is dense in $\R^+$.
\end{proposition}
\begin{proof}
By Theorem \ref{dense} we only need to show that $X$ is flow equivalent to shifts of arbitrarily large entropy.

	Let $N\in \N$ be given, set $M=2^N$, and assume that $\phi$ maps the symbols $0, 1$ of $\X_{[2]}$ to the symbols $a, b$ of $X$. Consider the words $w_1=ba, w_2=baa, \dots, w_M = b(a)^M$ and the corresponding different symbols $\diamondsuit_1, \dots, \diamondsuit_M\not\in \mathcal{A} (X)$. Since $w_i$ is non-overlapping for every $i$, we can use Proposition \ref{replOverlap} to form a sequence of shift spaces
\begin{equation*}
	X_M=X^{w_M\mapsto \diamondsuit_M}, X_{M-1}=X_M^{w_{M-1}\mapsto \diamondsuit_{M-1}}, \dots, X_1=X_2^{w_1\mapsto \diamondsuit_1}
\end{equation*}
such that $X\FE X_M \FE \dots \FE X_1$. Since all points that are concatenations of $w_i$'s belong to $X$, there is an embedding $\psi=\Psi_\infty: \X_{[M]}\to X_1$ generated by the 1-block map $\Psi:i\mapsto \diamondsuit_{i+1}$, and so $h(X_1)\geq h(\X_{[M]}) = \log M=N$ by Proposition \ref{entrOfFull} and Example \ref{fullREntropy}.
\end{proof}
\noindent
Given a shift space $X$ we wish to employ Proposition \ref{2shiftEmb} by  constructing a shift space $Y\FE X$ such that $Y$ contains a copy of the full 2-shift.

\begin{lemma}\label{stillSynchronising}
	Let $X$ be a shift space, and let the word $w\in B(X)$ satisfy that it is intrinsically synchronising for $X$ and that $X$ does not allow non-trivial $w$-overlaps. Then $\diamondsuit$ is intrinsically synchronising for the shift space $X\ex{w}{\diamondsuit}$.
\end{lemma}
\begin{proof}
	First of all, $Y=X\ex{w}{\diamondsuit}$ is a shift space by Lemma \ref{replOverlap}. 
	
	Now, consider words $u, v$ that $u\diamondsuit, \diamondsuit v\in B(Y)$. By construction of the space $Y$ we have $u\ex{\diamondsuit}{w} w,  wv\ex{\diamondsuit}{w}\in B(X)$ and since $w$ is intrinsically synchronising for $X$ it follows that $u\ex{\diamondsuit}{w} wv\ex{\diamondsuit}{w}\in B(X)$. Thus, 
	\begin{equation*}
		\left(u\ex{\diamondsuit}{w} \diamondsuit v\ex{\diamondsuit}{w}\right)\ex{w}{\diamondsuit} = u\diamondsuit v \in B(Y),
	\end{equation*} 
	and the result follows.
\end{proof}

\begin{theorem} \label{mainEntr}
	Let $X$ be a shift space and $Y\subseteq X$ be an irreducible subshift of $X$ with $h(Y)>0$.
 Assume further that there exists a word $w\in B(Y)$ such that $w$ is intrinsically synchronising for $Y$ and $X$ does not admit non-trivial $w$-overlaps. Then $h([X])$ is dense in $\R^+$. 
\end{theorem}
\begin{proof}
	Let $\mathcal{A}$ be the alphabet of $Y$. By Lemma \ref{replOverlap} and Lemma \ref{stillSynchronising} we can assume that $w$ is just a single letter $\diamondsuit\in \mathcal A$. Consider the set 
	\begin{equation*}
		A=\{ u\in B(Y) \mid \diamondsuit\not\in u \text{ and } \diamondsuit u\diamondsuit\in B(Y) \}
	\end{equation*}
 of words  that are ``sandwiched'' between two $\diamondsuit$'s. We will show that $\vert A\vert >1$. 
	
	First of all $A$ is non-empty. For by irreducibility, there exists $w\in B(Y)$ such that $\diamondsuit w\diamondsuit\in B(Y)$ and making $w$ of minimal length ensures that $\diamondsuit\not\in w$.
	
	Now, assume for contradiction that $\vert A\vert = 1$ and let $w$ be the single element of $A$. Then there is no word $u$ of $Y$ with $\vert u \vert >\vert w\vert$ such that $\diamondsuit$ is not a factor of $u$ , since else irreducibility of $Y$ would ensure words $s, t\in B(Y)$ with $\diamondsuit s u t \diamondsuit\in B(Y)$ and then some subword of $sut$ with length greater than $\vert w\vert$ would also belong to $A$. Thus, all points of $Y$ must be of the form $(\diamondsuit w)^\infty$, but this contradicts the assumption that $h(Y)>0$.
	
	Having proved that $\vert A\vert>1$ we can let $u, v\in A$ be different. Since $\diamondsuit$ is intrinsically synchronising for $Y$ every point of the form
	\begin{equation}\label{ab-descr}
		x = \dots \diamondsuit s_{-1}\diamondsuit s_{0}\diamondsuit s_1\dots, \text{ where } s_i\in \{u, v\}
	\end{equation}
	 belongs to $Y$. Assume without loss of generality that $\abs{u}\geq \abs{v}$. As $X$ allows neither non-trivial $\diamondsuit u$-overlaps nor non-trivial $\diamondsuit v$-overlaps because $u, v$ do not contain the letter $\diamondsuit$, we can use Lemma \ref{replOverlap} to construct shift spaces $X_1=X^{\diamondsuit u \mapsto a}$ and $X_2 = X_1^{\diamondsuit v \mapsto b}$ for different symbols $a, b\not\in\A$, such that
	\begin{equation*}
		X\FE X_1\FE X_2.
	\end{equation*}
	Using \eqref{ab-descr} we see that $X_2$ contains every point of $\{a, b\}^\Z$, so there is a 1-block code $\phi: \X_{[2]}\to X_2$, which is an embedding, and the result now follows by Proposition \ref{2shiftEmb}.
\end{proof}

\subsection{Classes of shift spaces with arbitrarily large entropy}
We shall now consider actual classes of shift spaces and apply the previous results to them. In all cases we show that the entropy of the flow equivalence class of the considered shift spaces is dense in $\R^+$.

\subsubsection{The S-gap shift and the $\beta$-shift}
First of all we consider the entropy of the $S$-gap shift as defined in Example \ref{sGapShiftEx}, which gives a nice example of the application of Theorem \ref{mainEntr}.
\begin{proposition}
	Let $S\subseteq \N$ with $\vert S\vert>1$. Then $h([\X(S)])$ is dense in $\R^+$.
\end{proposition}
\begin{proof}
	The result follows from Theorem \ref{mainEntr}, since $h(\X(S))>0$, $\X(S)$ is irreducible, and the symbol $1$ is both non-overlapping and  intrinsically synchronising for $\X(S)$.
\end{proof}
\noindent
Note that if $S$ only consists of a single element, then $h(\X(S)) = 0$. 

As we have not introduced the class of shift spaces called $\beta$-shifts, and doing so would be infeasible for this single appearance, we omit the proof of the following result. It is, however, almost exactly as for the $S$-gap shift.
\begin{proposition}
	For any $\beta>1$ let $\X_\beta$ be the associated $\beta$-shift. Then the set $h([\X_\beta])$ is dense in $\R^+$. 
\end{proposition}

\subsubsection{Sofic shifts}
For a more general result, we consider the class of sofic shifts.

\begin{theorem}\label{soficEntropy}
	Let $X$ be a sofic shift with $h(X)>0$. Then $h([X])$ is dense in $\R^+$. 
\end{theorem}
\begin{proof}
	Let $\mathcal G = (G, \mathcal L)$ be a minimal right-resolving representation of $X$. Since $h(\X_G)=h(\X_{\mathcal G})>0$ by Proposition \ref{soficEntropyByGraph} there is an irreducible subgraph $G'\subseteq G$ with $h(\X_{ G'})>0$ by Proposition \ref{irreducibleEntropy}.
	 Now, let $Y$ be the sofic shift presented by the labeled graph $\G'=(G', \Lab\restriction_{G'})$. Then $Y\subseteq X$ is sofic, irreducible, and $h(Y)>0$. So now for us to apply Theorem \ref{mainEntr} and complete the proof, we only need to find an intrinsically synchronising word $u$ for $Y$, such that $X$ does not allow non-trivial $u$-overlaps.  
	
		First, by Proposition \ref{replOverlap} we can replace any word $ab, a\neq b$ of $Y$ with a new symbol $\diamondsuit\not\in \A(X)$. The sofic shift $Y$ must contain a periodic point $v^\infty$, and by repeatedly replacing two adjacent and different letters of $v$ with new letters, we can achieve a shift $Z\FE Y$ containing a point $c^\infty$ for some symbol $c$. The shift space $Z$ is irreducible, sofic, and $h(Z)>0$, and we may let $\mathcal H$ be a presentation of $Z$. Then there is a vertex $I$ of $\mathcal H$, which is contained in a cycle where all edges are labeled $c$. We shall consider the word $v=c^k, k\in \N$ as though it is a single tour around the cycle beginning and ending at the node $I$. 
		
		Second, extend the word $v$ to the right from the vertex $I$ by way of the proof of Proposition \ref{extendToIntrin} to an intrinsically synchronising word for $Z$, $vw$, such that $w$ does not end with the letter $c$. This is possible since $Y$ is irreducible and $h(Y)>0$. Further, extend the word to the left to the intrinsically synchronising word $v^mw, m\in \N$ such that $\vert v^m\vert >\vert w \vert$. Then $v^mw$ is non-overlapping. To see this, let $c^n$ be the longest prefix of $w$ only containing $c$'s. Now, the only occurrence of $c^{km+n}$ in $v^mw$ is as the prefix. So if a proper prefix of $v^mw$ contains $c^{km+n}$ it cannot be a suffix. On the other hand, if a prefix of $v^mw$ does not contain $c^{km+n}$, then it consists entirely of $c$'s and cannot be a suffix, since $w$ ends in a letter different from $c$.
		
		As we have constructed a non-overlapping word $v^mw$ which is intrinsically synchronising for $Y$, we are done.
\end{proof}
\noindent
Since shifts of finite type are all sofic we immediately get the following corollary.

\begin{corollary}
	For a shift of finite type $X$ with $h(X)>0$, the set $h([X])$ is dense in $\R^+$.
\end{corollary}

\subsubsection{Non-periodic shift spaces}
When it comes to shift spaces without periodic points the methods we have used so far are insufficient. If $X$ is without periodic points there can be no embedding of the full 2-shift in a shift $Y\FE X$ since periodicity of points is maintained under flow equivalence and the full 2-shift contains periodic points. Instead we employ a different method resembling the one we shall see in Chapter 4. 

\begin{theorem}
	Let $X$ be a shift space without periodic points and $h(X)>0$. Then the set $h([X])$ is dense in $\R^+$.
\end{theorem}
\begin{proof}
Our approach will be to construct a shift space $Y\FE X$ with $h(Y)\geq \frac32 h(X)$. Since $Y$ is flow equivalent to $X$ it will not contain any periodic points, so the procedure can be repeated, and $X$ will be flow equivalent to shifts of arbitrarily large entropy. Applying Theorem \ref{dense} then completes the proof.

	 Let $\mathcal A = \{e_1, \dots, e_m\}$ be the alphabet of $X$, and $\mathcal F$ be a set of forbidden blocks for $X$. Since $X$ contains no periodic points it cannot contain any point of the form $e_i^\infty, 1\leq i\leq m$, so we may assume that there is a word $e_i^{k_i}\in \mathcal F, k_i\in \N$ for every $1\leq i\leq m$. 
	 
	 Create the collection of words
	 \begin{equation*}
	 	A = \left\{ se_i^l \mid s\in \mathcal A\setminus \{e_i\}, 1\leq l\leq k_i \right\},
	 \end{equation*}
	 and sort them by length such that $A=\{ w_1,\dots, w_r \}$ with $\vert w_i\vert \geq \vert w_j\vert$ for $i\geq j$.
	 
	 As $w_i$ is non-overlapping for all $i$, Proposition \ref{replOverlap} allows us, for different symbols $\diamondsuit_1, \diamondsuit_2, \dots, \diamondsuit_r\not\in \A$, to make the following word contractions.
	 \begin{equation*}
	 	X_1 = X^{w_1\mapsto \diamondsuit_1}, X_2 = X_1^{w_2\mapsto \diamondsuit_2}, \dots, X_r=X_{r-1}^{w_r\mapsto \diamondsuit_r},
	 \end{equation*}
	 such that
	 \begin{equation*}
	 	X\FE X_1\FE X_2\FE \dots\FE X_r.
	 \end{equation*}
	 Note that the order of the contractions is important as it facilitates that no word with two adjacent letters from the old alphabet $\A$ occurs in $X_r$.

	 Define functions 
	 \begin{equation*}
	 	T_i: B(X_{i-1})\to B(X_{i}), 1\leq i\leq r
	 \end{equation*}
	  such that $T_i(w) = w'$, which takes $w$; removes from the left and right end, respectively, the biggest possible blocks of the form $f^k, hg^l$ with $f, g, h\in \mathcal A, h\neq g$ and $k, l\in \N$; and thereafter replaces every occurrence of $w_i$ with the symbol $\diamondsuit_i$. The procedure is illustrated here, where each $s_i$ is a word and $f$ is not a prefix of $s_1$.
	 \begin{eqnarray*}
	 	w:&\overbrace{ff\dots ff}^k s_1\,w_i\,s_2\,w_i\,s_3\,h\overbrace{gg\dots gg}^l& \\
	 	&\downarrow  &\\
	 	w':& s_1\, \diamondsuit_i\, s_2\, \diamondsuit_i\, s_3&
	 \end{eqnarray*}
	 Note that this procedure never removes or changes any of the symbols $\diamondsuit_i$. 
	 
	 Construct the function $T:B(X)\to B(X_r)$ defined by 
	 \begin{equation*}
	 	T=T_r\circ T_{r-1}\circ \cdots\circ T_1,
	 \end{equation*}
	 and set $S:B(X_r) \to B(X)$ to be the function that takes $w'$ and expands every occurrence of $\diamondsuit_i$ to $w_i$. Letting $N=\max\{k_i\}$ we can see that $\abs{S(T(w))}\geq \vert w\vert-(2N+1)r$ as a total maximum of $2N+1$ symbols, all in $\A$, can vanish from the ends of $w$ under $T_i$.
	  
	 	 Consider the sets
	 	 \begin{align*}
	 	 	C_n &= \{T(w) \mid w\in B_n(X)\},\\
	 	 	D_n &= \{u\in C_n \mid \text{No other member of $C_n$ is a proper prefix of }u\},\\
	 	 	E_n(v) &= \{u\in C_n\mid v\text{ is a prefix of }u\}, v\in D_n.
	 	 \end{align*}
		Suppose that  $u\in D_n$ and let $T_n^{-1}$ be the pre-image of $T$ restricted to $B_n(X)$. Then every element of $T_n^{-1}(E_n(u))$ has the word $S(u)$ as a factor and length $n$, so we get the inclusion
		 \begin{equation*}
	T_n^{-1}(E_n(u)) \subseteq \{ w S(u) v\mid w\in B_k(\A^\Z), v\in B_l(\A^\Z), k+l+\abs{S(u)} = n \}.
		\end{equation*}
		 Using $\abs{S(u)}\geq n-(2N+1)r$ this yields
		 \begin{align*}
		 	\vert T_n^{-1}(E_n(u))\vert
		 		&\leq \sum_{k=0}^{n-\abs{S(u)}} \abs{B_k(\A^\Z)}\cdot \abs{B_{n-k-\abs{S(u)}}(\A^\Z)}\\
		 		&\leq ((2N+1)r+1)\vert \mathcal A\vert^{(2N+1)r}.
		 \end{align*}
This estimate, together with  
\begin{equation*}
	\bigcup_{u\in D_n} T_n^{-1}\left(E_n(u)\right)=T_n^{-1}(C_n)=B_n(X),
\end{equation*}
 yields
\begin{equation*}
	\vert D_n\vert\geq \frac{\abs{B_n(X)}}{R},
\end{equation*}
where $R=((2N+1)r+1)\vert \mathcal A\vert^{(2N+1)r}$ is a constant.
Any word of $D_{3n}$ is of length $\leq 2n$ since no two consecutive symbols in a word of $B(X_r)$ can be members of $\mathcal A$ and expanding any other symbol yields a word of length at least two. Further, as there is no $u\in D_{3n}$, which is a prefix of another $v\in D_{3n}$, and $D_{3n}\subseteq B(X_r)$, we have
\begin{equation*}
	\abs{B_{2n}(X_r)} \geq \abs{D_{3n}} \geq \frac{\abs{B_{3n}(X)}}{R}.
\end{equation*}
Since $R$ is simply a constant, we get
\begin{align*}
	\frac32 h(X) =   \frac32\lim_{n\to \infty} \frac1{3n}\log \vert B_{3n}(X)\vert
	     \leq\lim_{n\to\infty} \frac1{2n} \log  (R\cdot \vert B_{2n}(X_r)\vert) 
	     = h(X_r).
\end{align*}
\end{proof}

\chapter{Classification of the S-gap shift}
In this chapter we will attempt a classification of the $S$-gap shift with respect to flow equivalence. The classification is not fully achieved, but we show a new, stronger invariant for sofic $S$-gap shifts and classify the non-sofic $S$-gap shifts completely. Our approach will be to regard flow equivalence as a transformation of the orbits of a shift space and examine properties of these transformations.

\begin{remark}
Any result in this chapter that does not have a citation or is stated as known, is new and produced by the author.
Further, Proposition \ref{sGapFiniteClass} has an alternative proof supplied by the author.
\end{remark}

\section{Initial results on the S-gap shift}
The $S$-gap shifts were introduced in Example \ref{sGapShiftEx}, and this initial section will serve to present some already existing results on their classification. We show how to distinguish finite type, sofic, and non-sofic $S$-gap shifts, and then move on to classification with respect to flow equivalence. If $S\subseteq \N_0$ and $k\in \Z$ then $k+S$ will denote the set $\{k+s\mid s\in S\}$.

\subsection{Shift types}
First, we need to establish when an $S$-gap shift is sofic or of finite type.
\begin{proposition}[Dastjerdi and Jangjoo \cite{S-gap}]\label{sGapFiniteClass}
Let $S\subseteq \N_0$. Then $\X(S)$ is of finite type if and only if $S$ is finite or there exists $R\subseteq \N_0$ and $t\in \N_0$ with $S=R\cup (t+\N_0)$. 	
\end{proposition}
\begin{proof}
	If $S$ is finite Example \ref{sGapShiftEx} describes a finite set of forbidden words $\F$ with $\X_\F=\X(S)$. 
	If $S$ is infinite and $\X(S)$ is of finite type, then it is $M$-step for some $M$ by Proposition \ref{finiteMStep}. Since there is an $s\in S$ with $s>M$, then 
	\begin{equation*}
		\{10^k1\mid k>M\}\subseteq B(\X(S))
	\end{equation*}
	because none of these blocks can be forbidden in $\X(S)$. So $S$ must contain all $k>M$ and is thus of the form $S=R\cup (M+1+\N_0)$ for some $R\subseteq \N_0$.
	
	If $S$ is of the form $S=R\cup (t+\N_0)$, then a finite set of forbidden words for $\X(S)$ is given by 
	\begin{equation*}
	\F = \{10^r1\mid r\in \N_0\setminus S\}.
	\end{equation*}
\end{proof}

\begin{proposition}[Dastjerdi and Jangjoo \cite{S-gap}]\label{soficSGapShifts}
Let $S\subseteq	 \N_0$. Then $\X(S)$ is a sofic shift if and only if there are finite sets $R, T\subseteq \N_0$ and $N\in \N$ with $S = R\cup (T+N\N_0)$.
\end{proposition}
\begin{proof}
	The proof is omitted for brevity, but the interested reader is encouraged to come up with one herself!
\end{proof}
\noindent
When dealing with sofic $S$-gap shifts we will often want our representation $S=R\cup(T+N\N)$ to be of \emph{minimal form}. That is to say that $N$ is minimal, $R\cap (T+N\N_0)=\emptyset$, and $\max\{ \abs{t_i-t_j}\mid t_i, t_j\in T \}<N$. It is straightforward to realise that writing $S$ of the desired form is always possible for $\X(S)$ sofic.

\subsection{Beginning classification}
Second, we present already known results on the subject of classification with respect to flow equivalence.
\begin{lemma}[Johansen \cite{RuneJohansen}]\label{sGapRykRundt}
	Let $S\subseteq \N_0$.
	\begin{enumerate}
		\item If $k\in \N$, then $\X(k+S)\FE \X(S)$.
		\item If $a, b\in \N_0\setminus S$, then $\X(\{a\}\cup S)\FE \X(\{b\}\cup S)$.
	\end{enumerate}
\end{lemma}
\begin{proof}
	The first part is simply $k$ applications of Lemma \ref{extendWithExistingSymbol} since $\X(k+S)=\X(S)^{1 \mapsto 10^k}$.
	
	For the second part, assume without loss of generality that $b>a$ and consider the sliding block code $\phi=\Phi_\infty^{[0,a+1]}\colon \X(\{a\}\cup S)\to \{0, 1, \diamondsuit \}^\Z$ induced by the $(a+2)$-block map
	\begin{equation*}
		\Phi(w) = \begin{cases}
			      	  \diamondsuit, & w=10^a1\\
			      	  w_{[1]}, & w\neq 10^a1.
				  \end{cases}
	\end{equation*}
	The map $\phi$ is clearly injective, so $\X(\{a\}\cup S)$ is conjugate to the image $Y_1=\phi(\X(\{a\}\cup S))$. Now, use Lemma \ref{extendWithExistingSymbol} to make the expansion $Y_2=Y_1^{\diamondsuit \mapsto \diamondsuit 0^{b-a}}$. The sliding block code $\psi\colon Y_2 \to \X(\{b\}\cup S)$ induced by the $1$-block map 
	\begin{equation*}
		\Psi(w) = \begin{cases}
				      1, & w=1 \text{ or } w=\diamondsuit\\
				      0, & w=0.
				  \end{cases}
	\end{equation*}
	is a conjugacy with inverse induced by a $(b+2)$-block map similar to $\Phi$ since $\diamondsuit$ will always be followed by $0^b1$ in $Y_2$. Thus, 
	\begin{equation*}
			\X(\{a\}\cup S) \cong Y_1\FE Y_2 \cong \X(\{b\}\cup S).
	\end{equation*}
\end{proof}

\begin{proposition}[Johansen \cite{RuneJohansen}]\label{toMinimalForm}
	Let $S=\{r_1, \dots, r_k\}\cup(\{ t_1, \dots, t_l \}+N\N_0)$ be of minimal form with $t_1<\dots<t_l$. Furhter, let $1\leq j\leq l$, $j\equiv 1-k \pmod{l}$, and 
	\begin{equation*}
		S' = \{ 0, t_{j+1}-t_j, \dots, t_l-t_j, t_1+N-t_j, \dots, t_{j-1}+N-t_j \} + N\N_0,
	\end{equation*}
	then $\X(S) \FE \X(S')$.
\end{proposition}
\begin{proof}
	The proof is by multiple applications of Proposition \ref{sGapRykRundt} and is omitted for brevity.
\end{proof}
\noindent
This shows that we only need to consider sets of the form $S=T+N\N_0, T\subseteq \N_0$ in the classification problem for sofic $S$-gap shifts. 

The following gives a complete classification of the $S$-gap shifts of finite type. 
\begin{proposition}[Johansen \cite{RuneJohansen}]\label{finiteTypeSGapFlow}
Let $S\subseteq \N_0$ be given such that $\X(S)$ if of finite type. 
\begin{enumerate}
	\item If $S$ is finite, then $\X(S)$ is flow equivalent to the full $\abs{S}$-shift.
	\item If $S$ is infinite, then $\X(S)$ is flow equivalent to the full 2-shift.
\end{enumerate}	
\end{proposition}
\begin{proof}
	For the first part, let $S=\{r_1, r_2, \dots, r_n\}$ with $r_1<r_2<\dots<r_n$, and let $\diamondsuit_1, \dots, \diamondsuit_n$ be distinct new symbols. Using Lemma \ref{replOverlap} we now make word replacements to get the flow equivalent series of shift spaces
	\begin{equation*}
		X_1=\X(S)^{10^{r_n}\mapsto \diamondsuit_n}, X_2 = X_1^{10^{r_{n-1}}\mapsto \diamondsuit_{n-1}}, \dots, X_{n} = X_{n-1}^{10^{r_{1}}\mapsto \diamondsuit_{1}}.
	\end{equation*}
	Then $X\FE X_n = \{\diamondsuit_1, \dots, \diamondsuit_n\}^\Z \cong X_{[n]}$. Note  that replacing the longest strings of $0$'s first is required for this to hold.
	
	For the second part, we have $S=R\cup (t+\N_0)$ for some $R\subseteq \N$ and $t\in \N_0$ by Proposition \ref{sGapFiniteClass}. Proposition \ref{toMinimalForm} yields $\X(S)\FE \X(\N_0) = \X_{[2]}$.
\end{proof}
\noindent
In the literature, the strongest invariant for sofic $S$-gap shifts under flow equivalence is the following, which we will not prove here since a later section will contain a stronger result.

\begin{proposition}[Johansen \cite{RuneJohansen}]\label{bestForSoficYet}
	Let 
	$$S=\{s_1, \dots, s_k\}+N\N_0\,\, \text{ and }\,\, T=\{ t_1, \dots, t_l \} + M\N_0$$
	 be of minimal form. If $\X(S)\FE \X(T)$, then $k=l$ and $N=M$.
\end{proposition}

\section{Flow equivalence as transformations}
In order to classify the class of $S$-gap shifts further with respect to flow equivalence, we devote this section to the introduction of a new technique that has been developed by the author.

\begin{definition}
	Let $X$ be a shift space with shift map $\sigma_X$ and let $x\in X$ be a point. Then the \emph{orbit} of $x$ in $X$ is the set $\mathcal O(x)=\{ \sigma^k(x)\mid k\in \Z\}\subseteq X$.
\end{definition}
\noindent
We seek to view a flow equivalence as a function that maps orbits of one shift space bijectively to the orbits of another. The operations of conjugacy, symbol expansion, and symbol contraction are regarded as maps that we combine by composition into flow equivalences. Conjugacies already have functional representations, so we only need the following two definitions. 

\begin{definition}
	Let $X$ be a shift space, $a\in \mathcal A(X)$, $\diamondsuit\not\in \mathcal A(X)$, and $Y=X^{a\mapsto a\diamondsuit}$ such that $X\FE Y$. We represent this equivalence by an \emph{expansion}, a function $e:\mathcal O(X) \to \mathcal O(Y)$ given by
	\begin{equation*}
		e(\mathcal O(x)) = \mathcal O(x^{a\mapsto a\diamondsuit}),
	\end{equation*}
	where $x^{a\mapsto a\diamondsuit}$ is $x$ with every occurrence of $a$ replaced by $a\diamondsuit$.
	
	Further, for a word $w\in B(X)$ we write $e(w) = w^{a\mapsto a\diamondsuit}\in B(Y)$.
\end{definition}

\begin{definition}
	Let $X$ be a shift space; symbols $a, \diamondsuit\in \mathcal A(X)$ be given such that for $x\in X$, $x_{i+1}=\diamondsuit$ if and only if $x_{i}=a$; and write $X\FE X^{\diamondsuit\mapsto \epsilon}=Y$ by Lemma \ref{collapseTwoAlwaysConsecutiveLetters}. We represent this equivalence by a \emph{contraction}, a function $c:\mathcal O(X) \to \mathcal O(Y)$ given by
	\begin{equation*}
		c(\mathcal O(x)) = \mathcal O(x^{\diamondsuit\mapsto \epsilon}).
	\end{equation*}
	
	Further, for a word $w\in B(X)$, we write $c(w) = w^{\diamondsuit\mapsto \epsilon}\in B(Y)$. Note that since $a$ always comes before a $\diamondsuit$ in $X$ this operation is well defined.
\end{definition}
\noindent
Having these definitions in place we can now define flow equivalences as mappings between the orbits of shift spaces and show that such a representation is well-defined for all flow equivalences, though it need not be unique.

\begin{definition}
	Let $X, Y$ be shift spaces with $X\FE Y$. A \emph{functional representation of the flow equivalence} $X\FE Y$ is a bijective function $T\colon \mathcal O(X)\to \orb{Y}$ which is a composition of conjugacies, contractions, and expansions.
\end{definition}

\begin{proposition}
Every flow equivalence has a functional representation.	
\end{proposition}
\begin{proof}
Let $X\FE Y$ be shift spaces. Then by Theorem \ref{PaSu} there exists a series of shift spaces, $X_1, \dots, X_n$ such that $X= X_1$, $X_n= Y$ and for every $1\leq i<n$ either $X_{i+1}$ is a symbol expansion or contraction of $X_i$, or $X_i\cong X_{i+1}$. So there exists functions $f_1\dots, f_{n-1}$ that are either conjugacies, expansions, or contractions, which represent each equivalence, $X_i\FE X_{i+1}$. Thus, $T=f_{n-1}\circ \dots \circ f_1$ is a functional representation of the flow equivalence $X\FE Y$. Note that $T$ is bijective, since conjugacies, contractions, and expansions are all bijective when considered as functions from $\mathcal O(X)$ to $\mathcal O(Y)$.
\end{proof}

\begin{remark}
	For the remaining chapter, we will slightly abuse notation and use points and orbits of a shift space interchangeably in the context of functional representations of flow equivalences.
	\end{remark}

\noindent
Observe that when used on a word, a conjugacy or contraction may yield the empty word, $\epsilon$, but by choosing sufficiently long words we can ensure that this does not happen.

\begin{definition}
	Let $X\FE Y$ with the functional representation $T$. A word $w\in B(X)$ is called \emph{deciding for $T$} if $T(w)\neq \epsilon$. When $T$ is understood we simply say that $w$ is \emph{deciding}.
\end{definition}

\begin{proposition}
For every $X\FE Y$ with functional representation $T$ there is an $N$, such that  every $w\in B(X)$ with $\abs{w}\geq N$ is deciding for $T$.
\end{proposition}
\begin{proof}
	First note that $T=f_n\circ \dots \circ f_1$, where, for every $i$, $f_i$ is either a conjugacy, an expansion, or a contraction. Let $M$ be given such that all the conjugacies of $\{f_i\}$ are $m_i$-block codes with $m_i\leq M$.
	
	Now, for any word $w\in X$ a conjugacy, $f_i$, will satisfy $\vert f_i(w)\vert\geq \vert w\vert-M$; an expansion, $f_j$, will satisfy $\vert f_j(w)\vert \geq \vert w\vert$; and a contraction, $f_l$, will satisfy $\abs{f_l(w)}\geq \frac{\abs{w}-1}{2}$ since only every second symbol can be eliminated. Applying these estimates yields
	\begin{equation*}
		\abs{T(w)}=\abs{f_n\circ \dots \circ f_1(w)}\geq \frac{\abs{w}}{2^n}-(M+1)n,
	\end{equation*}
	and setting $N=((M+1)n+1)2^n$ finishes the proof.
\end{proof}
\noindent
It turns out that functional representations of flow equivalences have nice properties when applied to periodic words. These properties will later constitute the core of the analysis of the $S$-gap shift

\begin{definition}
	We say that a word, $w$, is \emph{non-repeating} if there is no $u$ with $w=u^t, t>1$.
\end{definition}

\begin{proposition}\label{periodicTrans}
Let $X\FE Y$ with the functional representation $T$. Further, let $u^\infty \in X$ for a non-repeating word $u$ and $u^k$ be deciding for $T$ for some $k>1$. Then
\begin{enumerate}
	\item $T(u^k) = v^{k'}v_{[1:l]}$ for a non-repeating $v\in B(Y)$ and some $k'\in \N_0, l\in \N$.
	\item $T(u^{k+1}) = v^{k'+1}v_{[1:l]}$, i.e., adding another repetition to the argument adds another repetition to the result.
\end{enumerate} 
\end{proposition}
\begin{proof}
	We will proceed by showing that for each of the operations conjugation, contraction, and expansion the lemma holds for blocks of the form $u^ku_{[1:l]}$. As $T$ is a composition of such operations, the proposition follows.
	
		\textbf{Conjugacies: } If the conjugacy $\phi\colon X\to Y$ is induced by an $M$-block map $\Phi$, $u\in B_n(X)$, and $k\in \N_0, l\in \N$, then $\Phi(u^ku_{[1:l]}) = v^{k'}v_{[1:l']}$ for some $v\in B_n(Y)$ and $k'\in \N_0, l'\in \N$. Since conjugacies maintain least periods, it also follows that $v$ must be non-repeating. Observing that $u^{k+1}u_{[1:l]}$ is $n$ symbols longer than $u^{k}u_{[1:l]}$ it follows that $\Phi(u^{k+1}u_{[1:l]})$ must be $n$ symbols longer than $v^{k'}v_{[1:l']}$ and hence, by periodicity, equal to $v^{k'+1}v_{[1:l']}$.
		
		 \textbf{Contractions: } Let $c\colon X\to Y$ be the contraction $Y=X^{\diamondsuit\mapsto \epsilon}$ and suppose that $a$ is the symbol that always precedes $\diamondsuit$ in $X$. Then $c(u^ku_{[1:l]}) = (u^{\diamondsuit\mapsto \epsilon})^k (u_{[1:l]})^{\diamondsuit\mapsto \epsilon}$ and clearly $c(u^{k+1}u_{[1:l]}) = (u^{\diamondsuit\mapsto \epsilon})^{k+1} (u_{[1:l]})^{\diamondsuit\mapsto \epsilon}$. Now, we only need to prove that $u^{\diamondsuit\mapsto \epsilon}$ is non-repeating. 
		 
		 Assume for contradiction that $u$ is non-repeating and $u^{\diamondsuit\mapsto \epsilon}=v^k$ for some word $v$ and $k>1$. We see that $u = (v^k)^{a\mapsto a\diamondsuit} = (v^{a\mapsto a\diamondsuit})^k$ yields a contradiction, so we must have $u \neq (u^{\diamondsuit\mapsto \epsilon})^{a\mapsto a\diamondsuit}$. This means that in $u$ there is an $a$ which does not precede a $\diamondsuit$ or a $\diamondsuit$ that does not have a preceding $a$. This is only possible, if the first and last symbol of $u$ are $\diamondsuit$ and $a$ respectively, because $u^k\in B(X)$ and $k>1$. So $u\diamondsuit$ must be the word $\diamondsuit (v^{a\mapsto a\diamondsuit})^k$, but then $u$ is also of the form $w^k, k>1$, and we have a contradiction.
		 				
		 \textbf{Expansions: }Let $e\colon X\to Y$ be the expansion $Y=X^{a\mapsto a\diamondsuit}$. Then $c(u^ku_{[1:l]}) = (u^{a\mapsto a\diamondsuit})^k (u_{[1:l]})^{a\mapsto a\diamondsuit}$ and clearly $c(u^{k+1}u_{[1:l]}) = (u^{a\mapsto a\diamondsuit})^{k+1} (u_{[1:l]})^{a\mapsto a\diamondsuit}$. We see that if $u$ is non-repeating, then so is $u^{a\mapsto a\diamondsuit}$, and the proof is complete.
\end{proof}

\noindent
Further, we need to incorporate this knowledge of the behaviour of periodic words under functional representations of flow equivalences in the context of specific points of a shift space. In the following the words that we use may be infinite in the sense that if $X$ is a shift space and $x$ is a point of $X$, we may write $x=vu$, where $v$ and $u$ are words extending infinitely to the left and to the right. We say that such words are \emph{left-infinite} and \emph{right-infinite}, respectively.

\begin{proposition}\label{rulesOfTrans}
	Let $X\FE Y$ with functional representation $T$, $v\in B(X)$ be deciding, and $T(v) = v'$.
	\begin{enumerate}
		\item If $vw\in B(X)$ for some word $w$, then $T(vw) = v'w'$ for some word $w'$.
		\item If $svt\in B(X)$ for words $s, t$, then $T(svt) = s'v't'$, $T(sv) = s'v'$, and $T(vt) = v't'$ for some words $s', t'$.
	\end{enumerate}
\end{proposition}
\begin{proof}
	All the postulates are true under conjugations, contractions, and extensions, so they must be true for a composition of these.
\end{proof}

\begin{proposition}\label{decidingSandwich}
	Let $X\FE Y$ with the functional representation $T$, $u^\infty \in X$ for a non-repeating word $u$, $u^k\in B(X)$ be deciding for a $k>1$, and $w, v\in B(X)$ such that $wu^kv\in B(X)$ and $wu^{k+1}v\in B(X)$. 
	
	If $T(u^k) = u'^{k'}u'_{[1:l]}$ for a non-repeating word $u'\in B(X)$ and some $k'>0$, then
	\begin{enumerate}
		\item  $T(wu^kv) = w'u'^{k'}v'$ for some $w', v'\in B(Y)$.
		\item  $T(wu^{k+1}v) = w'u'^{k'+1}v'$.
	\end{enumerate}
\end{proposition}
\begin{proof}
	By Proposition \ref{rulesOfTrans} (1) we have $T(wu^k) = w'u'^{k'}u'_{[1:l]}$ and $T(u^kv) = u'^{k'}v'$ for some words $w', v'$, where we choose to absorb $u'_{[1:l]}$ in $v'$. By Proposition \ref{rulesOfTrans} (2) we now have $T(wu^kv) = w'u'^{k'}v'$. The second part follows by observing that by Proposition \ref{periodicTrans}, $T(u^{k+1}) = u'^{k'+1}u'_{[1:l]}$.
\end{proof}

\begin{proposition}\label{periodicPoints}
	Let $X\FE Y$ with the functional representation $T$ and $u^\infty\in X$. Then $T(u^\infty) = v^\infty$ for some $v \in B(Y)$. 
\end{proposition}
\begin{proof}
	The proposition holds true for conjugations, contractions, and extensions, so it must hold for arbitrary  compositions of them.
\end{proof}

\section{Classification and a new invariant}
Existing results on the classification of $S$-gap shifts up to flow equivalence are based on graph invariants under flow equivalence and therefore focus almost solely on sofic shifts. This means that the results are no better than said graph invariants and that non-sofic shifts have not been classified.

This section will show, how stronger classification results can be achieved when we instead focus on the functional representations of the flow equivalences between $S$-gap shifts.

\begin{theorem}\label{newInvariant}
	Let $S, R$ be infinite subsets of $\N_0$ such that $\mathsf{X} (S)\FE \mathsf{X} (R)$. 
	Then there exists cofinite subsets $S'\subseteq S$ and $R'\subseteq R$ such that $S'+r=R'$ for some $r\in \Z$.
\end{theorem}
\begin{proof}
Let $X=\mathsf X(S)$ and $Y=\mathsf X(R)$ and $T$ be a functional representation of $X\FE Y$. Let further, by Proposition \ref{periodicPoints}, $T(0^\infty) = u^\infty$ for some non-repeating $u\in B(Y)$, and choose $k\in S, k>1$ such that $0^k$ is deciding with respect to $T$. 

Suppose that $x\in X$ is of the form $x=s10^k1t$ for infinite words $s, t$, the existence of such a point being clear. Then, by Proposition \ref{decidingSandwich} (1), we have
\begin{equation*}
	T(s10^k1t) = s'u^{k'}t'
\end{equation*}
for infinite words $s', t'$ and $k'\in \N_0$. Let $T^{-1}$ be the functional representation of $Y\FE X$ and the inverse of $T$. By Proposition \ref{decidingSandwich} (2) we can choose $k$ large enough that $u^{k'}$ is deciding with respect to $T^{-1}$ and such that $k'>1$. 

Now, we have two cases. Either $u=0$ or $u\neq 0$.
\\

If $u\neq 0$ and is non-repeating it must have 1 as a factor. So for every $r\in \N_0$ we have $s'u^{k'+r}t'\in Y$, as $u^2$ is allowed in $Y$, and by Proposition \ref{decidingSandwich}(2),
\begin{equation*}
	T^{-1}(s'u^{k'+r}t') = s10^{k+r}1t\in X.
\end{equation*}
Thus, $k+\N_0\subseteq S$ and therefore $X\FE \X_{[2]}$ by Propositions \ref{sGapFiniteClass} and \ref{finiteTypeSGapFlow}. As $X\FE Y$ it also follows that $Y\FE \X_{[2]}$ and $l+\N_0\subseteq T$ for some $l\in \N_0$. Setting $S'=k+\N_0$ and $R'=l+\N_0$ now yields the desired result.
\\

If $u=0$ we can let $k'$ be chosen such that $s'$ and $t'$ respectively ends and begins with a 1, so that $k'\in R$. To see that $s'$ and $t'$ both have 1 as a subword, just notice that if not, then either $T(s10^\infty)=0^\infty$ or $T(0^\infty 1 t)=0^\infty$, which contradicts the bijectivity of $T$. Now, for every $r, r'\in \N_0$ such that $k+r\in S$ and $k'+r'\in R$ we have $s10^{k+r}1t\in X$ and $s'0^{k'+r'}t'\in Y$, so by Proposition \ref{decidingSandwich}(2),
\begin{eqnarray*}
	&T(s10^{k+r}1t) &= s'0^{k'+r}t'\in Y,\\
	&T^{-1}(s'0^{k'+r'}t') &= s10^{k+r'}1t\in X
\end{eqnarray*}
Thus, for every $r\in \N_0$, $k+r\in S$ if and only if $k'+r\in R$. Setting $S' = \{ s\in S\mid s\geq k\}$ and $R'=\{t\in R\mid t\geq k'\}$ then yields $S'-k=R'-k'$, which finishes the proof.
\end{proof}
\noindent
This allows us to establish a stronger invariant on the sofic $S$-gap shifts than the one we referred to in Proposition \ref{bestForSoficYet}.
\begin{corollary}
		Let $\X(S_1)$ and $\X(S_2)$ be sofic shifts and $S_1=R_1\cup (T_1+N\N_0)$ and $S_2=R_2\cup(T_2+M\N_0)$ be written of minimal form. If $\X(S)\FE \X(T)$, then $N=M$ and there exists an $r\in \Z$ such that $r+T_1\equiv T_2 \pmod{N}$.
\end{corollary}
\noindent
Moving on, we will use the above result in a complete classification of the non-sofic $S$-gap shifts.

\begin{proposition}\label{ifPart}
	Let $S, R\subseteq \N_0$ be infinite and say that there exists cofinite subsets $S'\subseteq S$ and $R'\subseteq R$ such that $\vert S\setminus S'\vert=\vert R\setminus R'\vert$ and $S'+r=R'$ for some $r\in \Z$. Then $\mathsf{X} (S)\FE \mathsf{X} (R)$.
\end{proposition}
\begin{proof}
	Make the enumerations $S\setminus S' = \{s_1, \dots, s_n\}$ and $R\setminus R' = \{r_1, \dots, r_n\}$. Since $\X(S) \FE \X(S+r)$ by Proposition \ref{sGapRykRundt} we can assume that $R'=S'$ and that all elements of $R$ and $S$ are greater than $n$. Now, by Proposition \ref{sGapRykRundt} we have
	\begin{align*}
		\X(S'\cup \{1, \dots, n\}) &\FE \X(S'\cup \{s_1\}\cup\{2, \dots, n\})\\
		&\FE \X(S'\cup \{s_1, s_2\}\cup\{3, \dots, n\})\\
		& \hspace{1cm} \vdots \\
		&\FE \X(S'\cup \{s_1, \dots, s_n\})= \X(S).
	\end{align*}
	and similarly $\X(R'\cup \{1, \dots, n\})\FE \X(R)$, so $\X(S)\FE\X(R)$.
\end{proof}

\begin{lemma}\label{itIsPeriodic!}
	Let $T$ be an infinite subset of $\N_0$ and suppose that $r, r'\in \N_0$ with $r> r'$. If there exists an $N\in \N$ such that for $t>N$ we have $r+t\in T$ if and only if $r'+t\in T$, then $T=S\cup (R+(r-r')\N_0)$ for some finite  $S, R\subseteq \N_0$.
\end{lemma}
\begin{proof}
	For $q>N+r'$ we have $q\in T$ if and only if $q+(r-r')\in T$, and if $q>N+r$ we have $q\in T$ if and only if $q-(r-r')\in T$. So for every  $q>N+r'$ we have $q\in T$ if and only if there is a $q'\in T$ with $N+r'<q'\leq N+r$ and $q=q'+k(r-r')$ for some $k\in \N_0$.
	
	Setting $S = \{q\in T\mid q\leq N+r'\}$ and $R=\{q\in T\mid N+r'<q\leq N+r\}$ then yields $T=S\cup (R+(r-r')\N_0)$. 
\end{proof}

\begin{lemma}\label{savesMyAss}
	Let $S, S'\subseteq \N$ be infinite such that neither $\X(S)$ nor $\X(S')$ is flow equivalent to the full 2-shift, and $\X(S)\FE \X(S')$ with functional representation $T$. If $R\subseteq S$ is finite then there are no $x\in \X(R)$ with 
	\begin{equation*}
		T(x) = s0^\infty \,\,\, \text{ or } T(x) = 0^\infty s,
	\end{equation*}
	where $s$ and $t$ are left- and right-infinite words, respectively.
\end{lemma}
\begin{proof}
	Since $R$ is finite, $\X(R)$ is $M$-step for some $M\in \N$. Choose $M$ large enough that all words of length $M$ or greater are deciding for $T$, and note that all words of $B_{M+1}(\X(R))$ must have 1 as a factor, since else $0^\infty\in \X(R)$. Assume for contradiction and without loss of generality that $x\in \X(R)$ satisfies $T(x) = s0^\infty$ for some left-infinite word $s$. Then there is a right-infinite subword $x'$ of $x$ with $T(x') = \epsilon 0^\infty$, where the $\epsilon$ signifies that $T(x')$ is only right-infinite and not bi-infinite. The set $B_M(\X(R))$ is finite, so there must be a word $w\in B_M(\X(R))$ such that $wvw$ is a factor of $x'$ and $v\in B(\X(R))$ is not the empty word. Since $\X(R)$ is $M$-step we have $(wv)^\infty\in \X(R)$, and as $w$ is deciding for $T$ and $T(wvw)=0^k$ for some $k$, we have $T((wv)^\infty)=0^\infty$ by Proposition \ref{rulesOfTrans}. Recall from the proof of Theorem \ref{newInvariant} that if $\X(S)$ is not flow equivalent to the full 2-shift, then $y=0^\infty$ is the unique point satisfying $T(y) = 0^\infty$, but this contradicts that the point $(vw)^\infty\in \X(R)$ satisfies $T((vw)^\infty)=0^\infty$.
\end{proof}

\begin{theorem}
	Let $S, R\subseteq \N_0$ such that $\X(S)$ and $\X(R)$ are not sofic. Then $\mathsf{X} (S)\FE \mathsf{X} (R)$ if and only if there exists cofinite subsets $S'\subseteq S$ and $R'\subseteq R$ such that $\vert S\setminus S'\vert=\vert R\setminus R'\vert$ and $S'+r=R'$ for some $r\in \Z$. 
\end{theorem}
\begin{proof}
	If such cofinite subsets exist it follows directly from Proposition \ref{ifPart} that $\X(S)\FE \X(R)$.
	
	So instead assume that $\X(S)\FE \X(R)$. By Theorem \ref{newInvariant} there exists cofinite subsets $S'\subseteq S, R'\subseteq R$ such that $S'+r=R'$ for some $r\in \Z$, and as in the proof of Proposition \ref{ifPart}, we can assume that $S'=R'$.
	
		Let $T$ and $T^{-1}$ be the functional representations of the flow equivalences $\X(S)\FE \X(R)$ and $\X(R)\FE \X(S)$, respectively. We may assume that all elements of $S'$ and $R'$ are greater than the elements of $S\setminus S'$ and $R\setminus R'$ respectively, and that for all $s\in S'$ and $r\in R'$, $0^s$ and $0^r$ are deciding with respect to $T$ and $T^{-1}$ respectively. We will show that $T$ restricted to $\X(S\setminus S')$ maps bijectively to $\X(R\setminus R')$ giving a flow equivalence $\X(S\setminus S')\FE \X(R\setminus R')$ and proving $\vert S\setminus S'\vert=\vert R\setminus R'\vert$ by Proposition \ref{finiteTypeSGapFlow} and Example \ref{fullShiftsAreDifferent}. 
		
		Assume for contradiction that there is a point $x\in \X(S\setminus S')$ such that $T(x)\not\in \X(R\setminus R')$. Then $T(x) = v10^r1w$
		for some infinite words $v, w$  and $r\in R'$ by Lemma \ref{savesMyAss}. As in the proof of Theorem \ref{newInvariant}, we see that since $0^r$ is deciding with respect to $T^{-1}$ and $\X(S)$ is not flow equivalent to the full 2-shift, we get $T^{-1}(v10^r1w)= v'10^{r'}1w'$ for some $r'\in S\setminus S', r'>0$. Further, for every $t\in \N_0$ with $r+t\in R'$ we have $T^{-1}(v10^{r+t}1w)= v'10^{r'+t}1w'\in X(S)$, so $r'+t\in S$. Since the set $\abs{S\setminus S'}$ is bounded, there is an $N$ such that  $t\geq N$ and $r+t\in R'$ will ensure $r'+t\in S'$. Then $0^{r'+t}$ is deciding for $T$, so for $t\geq N$ and $r'+t\in S'$ or $r+t\in R'$ we will have
\begin{align*}
	T(v'10^{r'+t}1w') &= v10^{r+t}1w\in \X(R),\\
	&\text{ or }\\
	T^{-1}(v10^{r+t}1w) &= v'10^{r'+t}1w'\in \X(S)
\end{align*}		
respectively. Thus, $r+t\in R'$ if and only if $r'+t\in S'$ for $t>N$. Since $r\neq r'$ as $r'\in S\setminus S'$ and $r\in R'=S'$; $S\setminus S'$ and $R\setminus R'$ are finite; and $R'=S'$, it follows by Lemma \ref{itIsPeriodic!} that $R$ and $S$ can be written of the form $A\cup (B+\abs{r-r'}\N_0)$ for finite $A, B\subseteq \N_0$. Thus, $\X(S)$ and $\X(R)$ are sofic by Proposition \ref{soficSGapShifts}, a contradiction.
Thus, we have $T(\X(S\setminus S' ))\subseteq \X(R\setminus R')$, and by a symmetrical argument $T^{-1}(\X(R\setminus R' ))\subseteq \X(S\setminus S')$. As $T$ is a bijection it must then map $\X(S\setminus S')$ bijectively to $\X(R\setminus R')$.
\end{proof}

\noindent
As a final note, it seems natural given Theorem \ref{newInvariant} and Proposition \ref{ifPart} to conjecture that we can also include sofic shift in the above theorem.

\begin{conjecture}
	Let $S, R\subseteq \N_0$ be infinite. Then $\mathsf{X} (S)\FE \mathsf{X} (R)$ if and only if there exists cofinite subsets $S'\subseteq S$ and $R'\subseteq R$ such that $\vert S\setminus S'\vert=\vert R\setminus R'\vert$ and $S'+r=R'$ for some $r\in \Z$. 	
\end{conjecture}

\nocite{ParrySullivan}
\nocite{Franks}
\nocite{RuneJohansen}
\nocite{S-gap}
\vfill
\bibliographystyle{siam}
\bibliography{bibliography}

\begin{thebibliography}{1}

\bibitem{S-gap}
{\sc D.~A. Dastjerdi and S.~Jangjoo}, {\em Dynamics and topology of $s$-gap
  shifts}, Topology Appl., 159 (2012), pp.~2654--2661.

\bibitem{Franks}
{\sc J.~Franks}, {\em Flow equivalence of subshifts of finite type}, Ergodic
  Theory Dynam. Systems, 4 (1984), pp.~53--66.

\bibitem{RuneJohansen}
{\sc R.~Johansen}, {\em On flow equivalence of sofic shifts}, PhD thesis,
  University of Copenhagen, Copenhagen, 2011.

\bibitem{LM}
{\sc D.~Lind and B.~Marcus}, {\em An Introduction to Symbolic Dynamics and
  Coding}, Cambridge University Press, 1999.

\bibitem{ParrySullivan}
{\sc B.~Parry and D.~Sullivan}, {\em A topological invariant of flows on
  1-dimensional spaces}, Topology, 14 (1975), pp.~297--299.

\bibitem{FiniteClassification}
{\sc R.~F. Williams}, {\em Classification of subshifts of finite type}, Annals
  of Math., 98 (1973), pp.~120--153.

\end{thebibliography}
\addcontentsline{toc}{chapter}{Bibliography}

\end{document}